\documentclass[10pt,twoside]{article}
\usepackage[utf8]{inputenc}
\usepackage{amssymb}
\usepackage{float}
\usepackage{algorithm}
\usepackage{algorithmic}
\usepackage{ragged2e}
\usepackage{dsfont}
\usepackage[unicode]{hyperref}
\usepackage{amsfonts}
\usepackage[centertags]{amsmath}
\usepackage{amsthm}
\usepackage{newlfont}
\usepackage{array}
\usepackage{mathtools}
\usepackage{graphicx}
\usepackage{calrsfs}
\usepackage{xcolor}
\usepackage{multirow}
\usepackage[mathscr]{euscript}

\begin{document}

\centerline{\Large{\bf Fermatean fuzzy type entropy-based new integrated decision making method }}
\centerline{\Large{\bf  with analysis of energy poverty in T\"{u}rkiye application}}
%\centerline{\Large{\bf    based on WASPAS and similarity measures}}

\centerline{}

\centerline{\bf {Halim Ba\c{s}}(a), \bf {Murat Kiri\c{s}ci}(b*)}
%\centerline{}
\centerline{(a) Marmara University, Faculty of Economics, Istanbul, T\"{u}rkiye}
\centerline{(b) Istanbul University-Cerrahpa\c{s}a, Department of Biostatistics and Medical Informatics, Istanbul, T\"{u}rkiye}

\centerline{e-mail: mkirisci@hotmail.com}

\centerline{*Corresponding Author}
%\centerline{}

\newtheorem{Theorem}{\quad Theorem}[section]

\newtheorem{Definition}[Theorem]{\quad Definition}

\newtheorem{Proposition}[Theorem]{\quad Proposition}

\newtheorem{Corollary}[Theorem]{\quad Corollary}

\newtheorem{Lemma}[Theorem]{\quad Lemma}

\newtheorem{Example}[Theorem]{\quad Example}

\centerline{}
\textbf{Abstract:}
Income-based poverty indicators are insufficient to explain the multifaceted socio-economic issue of energy poverty fully. The ability of households to obtain and use energy services sustainably is significantly influenced by structural factors like building efficiency, energy prices, climate, and regional disparities. Moreover, the assessment of energy poverty inherently involves uncertainty, vagueness, and subjective judgments, which limit the applicability of classical deterministic evaluation approaches. This paper proposes a novel, integrated multi-criteria decision-making framework based on Fermatean fuzzy sets to address these issues. To objectively assess the uncertainty present in Fermatean fuzzy information, a novel Fermatean fuzzy entropy metric is first presented. The proposed entropy employs a nonlinear structure, enabling a more flexible and sensitive representation of fuzziness in complex decision environments. The subjective weights obtained from the Fermatean fuzzy PIPRECIA approach are then blended with the entropy-based objective weights. Lastly, options are ranked according to their relative usefulness in relation to ideal and anti-ideal solutions using the Fermatean fuzzy MARCOS approach. The proposed framework is applied to the regional evaluation of energy poverty in T\"{u}rkiye, considering seven geographical regions and six key criteria: income level, energy prices, energy efficiency, building efficiency, climate, and urbanization. The results reveal significant regional disparities in energy poverty, with Eastern and Southeastern Anatolia exhibiting the highest vulnerability, while Marmara and Aegean regions display the lowest risk. Sensitivity analysis, weight perturbation analysis, and entropy-criterion dominance analysis validate the robustness and stability of the suggested model across various parameter settings. The results illustrate the multifaceted and regionally distinctive nature of energy poverty, offering crucial theoretical, managerial, and policy-relevant insights. For policymakers and practitioners seeking to develop focused, effective, and socially just energy policies in the face of uncertainty, the proposed method provides a robust and adaptable decision-support tool. 
\centerline{}

%{\bf Subject Classification:} Primary *** ; Secondary ****. \\


\section{Introduction}

Due to rising energy costs, climate change, and widening social disparities, energy poverty has become a growing concern in both developed and emerging nations. Energy poverty, in contrast to traditional income-based poverty, is the result of households' inability to obtain sufficient energy services for lighting, heating, cooling, and other essential needs. This phenomenon is intrinsically multifaceted and context-dependent, influenced not only by financial limitations but also by structural, climatic, and regional factors.\\

Assessing energy poverty poses substantial methodological challenges. Decision-makers must account for heterogeneous criteria, uncertain and incomplete information, and subjective expert judgments. Fuzzy set theory and sophisticated multi-criteria decision-making (MCDM) techniques have gained popularity since traditional quantitative approaches often fall short of capturing these intricacies. Compared to intuitionistic and Pythagorean fuzzy sets, Fermatean fuzzy sets (FFSs) offer a wider domain for membership and non-membership degrees, providing a more flexible framework for characterizing uncertainty, as seen in recent advances in fuzzy theory.\\

Within this context, entropy-based weighting methods play a vital role in objectively determining the relative importance of criteria under uncertainty. However, existing entropy measures are often insufficiently sensitive to nonlinear and complex fuzzy information structures, particularly in real-world socioeconomic applications such as energy poverty.\\

This study goals to address these methodological and applied challenges by proposing a novel Fermatean fuzzy entropy-based integrated MCDM framework and applying it to the regional assessment of energy poverty in T\"{u}rkiye.
% -----------------------------------------------------------

\subsection{Research Motivation}
This study is primarily motivated by two interconnected needs. First, there is a growing understanding that energy poverty is a complex socio-technical phenomenon driven by elements such as housing quality, climate, energy efficiency, and geographical inequality, rather than being solely an income-related problem. Second, the subjectivity, ambiguity, and uncertainty inherent in expert-based assessments of such multifaceted issues are often too great for current decision-support technologies to handle.\\

T\"{u}rkiye presents a particularly relevant case due to its pronounced regional disparities in income, climate, building stock, and urbanization patterns. Under a uniform national energy policy, regions experience markedly different levels and forms of energy poverty. This motivates the need for a robust, uncertainty-aware analytical framework that can capture regional heterogeneity and support targeted policy interventions.

\subsection{Necessity}

From a methodological perspective, traditional MCDM approaches based on crisp or probabilistic data are inadequate for modeling the ambiguous and linguistic nature of expert judgments in energy poverty assessments. While fuzzy-based MCDM methods have been widely used, many existing studies rely on intuitionistic or Pythagorean fuzzy sets, which impose restrictive constraints on membership and non-membership degrees.\\

From a policy and managerial perspective, the absence of region-specific and multidimensional assessments of energy poverty limits the effectiveness of energy and social policies. Without a reliable prioritization of regions and criteria, public resources risk being allocated inefficiently, and one-size-fits-all policy approaches may exacerbate existing inequalities. Therefore, a comprehensive and uncertainty-sensitive evaluation framework is urgently needed.

\subsection{Research Gap}
A review of the current literature reveals several significant gaps. First, despite the growing interest in Fermatean fuzzy sets in recent years, there remains a dearth of entropy metrics designed explicitly for Fermatean fuzzy data. Existing entropy formulations often lack sufficient nonlinearity and sensitivity to capture complex uncertainty structures.\\

Second, there is a lack of integrated frameworks that simultaneously combine objective entropy-based weighting, subjective expert-driven weighting, and robust ranking methods under a unified Fermatean fuzzy environment. Most studies focus on either methodological developments or isolated applications, rather than offering a comprehensive decision-support framework.\\

Third, empirical studies on energy poverty predominantly focus on national-level indicators or household microdata, while sub-national and regional analyses—especially under uncertainty—remain scarce. In the context of T\"{u}rkiye, systematic regional energy poverty assessments using advanced fuzzy MCDM techniques are largely absent.

\subsection{Originality}

The originality of this study lies in several aspects. First, a novel Fermatean fuzzy entropy measure is proposed, incorporating a nonlinear structure that enhances sensitivity to uncertainty and fuzziness. This entropy measure satisfies essential axiomatic properties and provides a flexible tool for objective criteria weighting.\\

Second, the study introduces an integrated FF-MCDM framework that combines the proposed entropy with FF-PIPRECIA for subjective weighting and FF-MARCOS for alternative ranking. To the best of the authors' knowledge, this comprehensive methodology has never been used in an energy poverty study before.\\

Third, applying the suggested methodology to the regional evaluation of energy poverty in T\"{u}rkiye is a new empirical contribution that sheds light on regional determinants of energy poverty as well as geographical disparities.

\subsection{Contribution}
The following succinctly describes the research's contributions: \\

Theoretical Contribution: By presenting a novel FF-entropy measure and demonstrating its efficacy in addressing challenging socio-economic decision-making issues, the paper contributes to the body of knowledge in fuzzy decision theory. Additionally, it supports the idea that energy poverty is a multifaceted, regionally specific phenomenon. \\

Methodological Contribution: An integrated Fermatean fuzzy MCDM framework combining entropy, FF-PIPRECIA, and FF-MARCOS is proposed, providing a robust and flexible tool for decision-making under uncertainty.\\

Empirical Contribution: The regional application to T\"{u}rkiye reveals significant spatial disparities in energy poverty and identifies key criteria driving vulnerability in different regions.\\

Policy and Managerial Contribution: The results provide actionable insights for policymakers and energy managers, enabling targeted, region-specific interventions and supporting the design of socially just and efficient energy policies.

\section{Literature}

\subsection{Energy Poverty}

The literature defines energy poverty as a multifaceted condition that cannot be fully explained by inadequate income. Early research has shown that energy poverty should be considered in conjunction with factors such as energy efficiency, home characteristics, and access to energy services. It is not just about households' ability to pay for energy expenses. Boardman \cite{board}, and Hills \cite{hills} highlighted the flaws in conventional methods of quantifying fuel poverty, demonstrating how statistical criteria that mask the actual level of poverty can lead to households in low-energy-efficiency dwellings being overlooked. Bouzarovski and Petrova \cite{bour} and Thomson et al. \cite{thompson}, who contended that energy poverty is not just an economic reality but also intricately linked to spatial and social inequities, expanded on this strategy. As a result, the literature on energy justice, welfare regimes, and social policy all touch on the topic of energy poverty.\\

Among the most important factors influencing energy poverty are income level and the dynamics of energy prices. Particularly during times of price shocks, not having enough money to pay for energy costs exacerbates energy poverty and lowers household welfare. Okushima \cite{okush} demonstrates how growing energy costs and inadequate energy efficiency disproportionately affect low-income households in Japan, using a novel indicator. Renner et al. \cite{renner} draw attention to the relationship between price fluctuations and welfare disparities, stressing that in Indonesia, rising energy costs disproportionately affect low-income households. Consequently, interventions such as the removal of subsidies should be carefully designed and implemented. Guan et al. \cite{guanetal}, comparatively examining the pressure of price crises on household budgets on a global scale, also reveal that rapid price increases can make energy costs 'unaffordable', thus deepening energy poverty. Belaïd \cite{bela} links this debate to climate and carbon policies, arguing that policies targeting emission reduction can increase energy poverty by having unexpected negative consequences for low-income groups. Therefore, the dimension of social equity should be central to the design of policy. In a broader causal framework, Cheikh et al. \cite{chei} emphasize that the primary factors triggering energy poverty in Europe are energy price increases, low income levels, and energy-inefficient housing, highlighting the need for targeted policies that account for regional differences. In this context, the literature emphasizes that pricing and income support are crucial in combating energy poverty, but are insufficient on their own and should be considered in conjunction with structural factors, such as housing and energy efficiency.\\

Among the main factors influencing the structural aspect of energy poverty are dwelling features and energy efficiency. To reduce energy poverty, Abbas et al. \cite{abbas} argue that expanding energy access and enhancing energy and environmental efficiency in developing Asian nations are crucial, and that investments in renewable energy should complement this process. Income inequality, energy poverty, and energy efficiency are cyclically related, according to Dong et al. \cite{dong}, who argue that inequality exacerbates energy poverty and hinders households' ability to make energy efficiency investments. Conversely, low efficiency exacerbates energy poverty. Regarding the institutional capacity and governance aspect, Wang et al. \cite{wang} propose that fiscal decentralization can enhance local governments' ability to combat energy poverty, but this is contingent upon the utilization of technological advancements and energy efficiency. In a rural context, Liu et al. \cite{liuzhou} emphasize that rural centralization policies in China can increase access to energy. However, this effect will remain limited unless supported by infrastructure and efficiency improvements.\\

Evans et al. \cite{evans}, with their conceptualization of the “poverty premium,” emphasize that low-income households are often forced to live in low-energy housing and that increasing energy expenditures reinforces poverty. Housing renovations and reorganization processes can mitigate vulnerability by reducing energy costs through the use of efficiency resources, as Boemi and Papadopoulos \cite{boemi} suggest in Greece. Meanwhile, Matos et al. \cite{mafaldo} note that thermal building regulations and energy efficiency policies play a crucial role in mitigating energy vulnerability in Portugal. Desvallées \cite{desva} reveals that low-carbon renovations in Southern European social housing not only reduce consumption but also improve thermal comfort and quality of life. Fabbri and Gaspari \cite{fabbari} map the spatial appearance of energy vulnerability using energy performance certificates in Bologna, showing that low-performing buildings are unevenly clustered within the city. Teixeira et al. \cite{teix} provide a multi-dimensional framework that combines this exchange, ensuring that housing quality, energy efficiency, physical accessibility, and thermal comfort are also measured continuously and systematically alongside income and prices.\\

The literature increasingly highlights that energy poverty is not limited to domestic heating or electricity use, but intersects with areas such as climate vulnerability, health and education impacts, transportation, and technological transformation. Robinson and Mattioli \cite{robinson} address energy poverty in the context of the United Kingdom as a two-pronged vulnerability, analyzing both the lack of access to energy in homes (fuel poverty) and transportation energy poverty, while emphasizing how the geographical and economic disadvantages of low-income households exacerbate both areas. In the discussion of technological solutions, Sovacoll and Furszyfer Del Rio \cite{sova} acknowledge the potential of smart home technologies in Europe to optimize consumption, reduce costs, and increase efficiency. However, they note that financial barriers to accessing these technologies and data security risks can create new inequalities for low-income households. In the area of measurement and estimation, Al Kez et al. \cite{alkez} demonstrate the determinants of income, energy consumption, and housing characteristics by estimating energy poverty in the United Kingdom using machine learning methods, arguing that data-driven approaches can strengthen targeting in policy design. In terms of impact channels, Katoch et al. \cite{katoch} systematically examine the consequences of energy poverty on health and education, highlighting effects such as non-communicable diseases, mental health problems, and impaired educational performance of children, particularly in low-income regions. In the climate dimension, Casillas and Kammen \cite{casi} emphasize that reducing energy poverty is crucial for both development and combating climate change; renewable energy can increase access, reduce poverty, and limit emissions. Thomson and Snell \cite{thompsnell} suggest that EU and national policy frameworks addressing climate change can serve as a starting point for addressing energy poverty, and that this area can be intersected with the energy efficiency agenda. Discussing this tense relationship, Chakravarty and Tavoni \cite{chakravarty} argue that increasing consumption to reduce energy poverty can lead to higher emissions; however, these two goals can be harmonized through the use of renewable energy and efficiency technologies. Jessel et al. \cite{jessel} present an intersection of energy poverty, health, and climate, drawing on a broad literature, while Thomson et al. \cite{thompsenetal}, Streimikiene et al. \cite{streim} emphasize that energy poverty should also encompass the need for cooling and that rising temperatures exacerbate public health risks. Streimikiene et al. \cite{streim} state that if climate policies targeting households in the EU are designed in conjunction with social assistance and efficiency measures, both emission reduction and energy poverty reduction can be achieved simultaneously. Ehsanullah et al. \cite{ehsan} examine how energy insecurity evolves into energy poverty in conjunction with environmental concerns, while Yadava et al. \cite{yadava} emphasize that the interplay of energy, poverty, and climate vulnerability has combined effects on sustainable development, highlighting the need for coordinated policy packages. This expanding literature shows that energy poverty cannot be explained solely by the expenditure-to-income ratio, but is a state of vulnerability that encompasses multiple risks.\\

Urbanization and spatial inequalities are important contextual dynamics that transform both the determinants and manifestations of energy poverty. Sheng, He and Gua \cite{sheng} emphasize that urbanization, with appropriate infrastructure and planning, can offer a more energy-efficient lifestyle; however, rapid and unplanned urbanization can lead to social and environmental problems due to infrastructure deficiencies and increased demand. Sun and Tong \cite{sun} argue that rapid urbanization exacerbates energy poverty in low-income households by increasing pressure on energy infrastructure, and that regional inequalities in access to energy in cities further deepen the problem. Reames \cite{reames} address the urban inequality dimension from a broader social justice perspective, stating that low-income and ethnic minority groups are disproportionately affected by energy costs due to their concentration in energy-inefficient housing, and that energy poverty reflects social inequalities. Similarly, Yawale et al \cite{yawele} demonstrate that income inequality, inadequate access to infrastructure, and high costs collectively exacerbate urban energy poverty in Indian metropolises. Discussing the specific trajectory of urbanization dynamics in developing countries, Mahumane and Mulder \cite{mahu} highlight that in Mozambique, rapid urbanization has created a "new" problem area of urban energy poverty due to the pressure on energy infrastructure, and that low-income households migrating from rural to urban areas face difficulties in accessing energy. Hasanujzaman and Omar \cite{hasanuj} reveal that, in the case of Bangladesh, both intra-household and extra-household factors affect multidimensional energy poverty. They highlight that urbanization can produce both mitigating and amplifying effects, emphasizing the need for policies that are sensitive to local dynamics. In a broader international context, Xu et al. \cite{xuetal} note that, in the 48 countries covered by the BRI, urbanization can reduce rural energy poverty but also negatively impact environmental sustainability by increasing carbon emissions. Qi et al. \cite{qiChen} acknowledge the potential of urbanization in China to reduce energy poverty by increasing access to energy infrastructure, but emphasize that regional disparities and rapidly growing urban energy demand may limit this positive effect. In the context of the EU, Thomson and Bouzarovski \cite{thompsonBou} demonstrate that energy poverty deepens social inequalities and necessitates an approach that integrates efficiency improvements, social assistance, and pricing policies, while Bouzarovski and Simcock \cite{bouz} address energy justice from a spatial perspective, showing how infrastructure inequalities and geographical factors shape access to energy. This accumulation of work highlights the risk that analyses that overlook the spatial patterns of energy poverty risk may overlook problematic regions and vulnerable groups.\\

In light of these findings, two main gaps stand out in the literature: (i) although the need for multidimensional measurement is frequently emphasized (Hills, \cite{hills}; Thomson et al., \cite{thompson}; Teixeira et al., \cite{teix}), MCDM techniques that allow decision-makers to weight and rank multiple criteria simultaneously remain relatively limited; (ii) while existing MCDM studies are increasing, they often fail to discuss the country-specific policy context thoroughly or are not systematically widespread in developing countries. In this regard, Zhou et al. \cite{zhouwang} demonstrate the comparative measurement power of MCDM by assessing energy poverty in BRI countries using the entropy weight-TOPSIS approach, while Lu et al. \cite{LuRen} emphasize the importance of spatial-temporal energy poverty analysis at the sub-national level in China. Hasheminasab et al. While \cite{hashem} demonstrates the contribution of multidimensional comparison to policy discussions by proposing a new assessment for EU countries, Shieh and Shah \cite{sheih} highlight the measurement advantages, especially under data uncertainty, by developing a fuzzy MCDM-based multidimensional index for developing countries. Nevertheless, in a developing country like T\"{u}rkiye, there are limited studies that produce a holistic index or assessment using MCDM techniques, considering the dimensions of income, price, efficiency, climate, and urbanization together. This study aims to fill this gap.

\subsection{Fermatean Fuzzy Environment}

Real-life decision problems are inherently multifaceted and often involve a substantial degree of uncertainty. This uncertainty may arise from imprecise information, randomness, incomplete datasets, or the subjective perceptions of decision-makers. In many practical applications, problem attributes are commonly expressed using linguistic terms rather than precise numerical values. Capturing and modeling such imprecise knowledge is essential for improving the reliability of decision outcomes and supporting more effective solution strategies.\\

Fuzzy set (FS) theory, initially introduced by Zadeh \cite{Zadeh}, provides a mathematical framework for representing uncertainty and vagueness in human-centric systems. Although FS theory marked a significant departure from classical binary logic by allowing partial membership, it does not explicitly reflect the hesitation or dissatisfaction inherent in human judgments. To overcome this limitation, Atanassov \cite{Atan} proposed intuitionistic fuzzy sets (IFSs), which incorporate both membership and non-membership degrees. Since their introduction, IFSs have been extensively applied in optimization and decision-analysis studies.
However, IFS theory is constrained when the sum of membership and non-membership degrees exceeds unity. To relax this restriction, Yager \cite{Yager0}, \cite{Yager} introduced Pythagorean fuzzy sets (PFSs), where the squared sum of membership and non-membership degrees is required to be less than one. Building upon these developments, Senapati and Yager \cite{SenYager}, \cite{SenYager1} proposed Fermatean fuzzy sets (FFSs), which further extend the expressive capacity of fuzzy modeling. In FFSs, the sum of the membership and non-membership degrees raised to the third power is bounded by one, enabling a richer representation of uncertainty. Numerous applications of FFSs have since been reported in the literature \cite{gargetal, Jeevaraj}, \cite{kirisci0}- \cite{LiuLiu}, \cite{SenYager1, SenYager2, kirisci22}.\\

The representation of ambiguity and vagueness in decision environments was initially addressed through the fuzzy set framework introduced by Zadeh \cite{Zadeh}. While classical FSs capture uncertainty via membership degrees, they fail to account for explicit non-membership information. To address this shortcoming, Atanassov’s intuitionistic fuzzy set (IFS) theory \cite{Atan}  incorporates non-membership degrees, thereby providing a more informative description of evaluation data. Extensions of IFS theory have been explored in various mathematical and applied contexts, including the work of Kirisci \cite{kirisci2019}, who introduced Fibonacci statistical convergence in intuitionistic fuzzy normed spaces.\\

Despite their advantages, IFSs impose a strict constraint on the sum of membership and non-membership degrees, limiting their applicability in highly uncertain environments. In response, Yager \cite{Yager0}, \cite{Yager} proposed the Pythagorean fuzzy set (PFS) model, which allows greater flexibility by enforcing the condition $MD^{2} + ND^{2} \leq 1$. This relaxation enables decision-makers to articulate their preferences more freely. Nevertheless, the growing complexity of decision-making structures continues to challenge experts in generating consistent and reliable assessments.\\

To further enhance the expressive power of fuzzy models, Fermatean fuzzy sets (FFSs) were introduced by Senapati and Yager \cite{SenYager}. By constraining the cubic sum of membership and non-membership degrees, FFSs provide a broader domain for information representation compared to both IFSs and PFSs. The foundational properties and mathematical characteristics of FFSs were subsequently established in \cite{SenYager1}, \cite{SenYager2}, demonstrating their superiority in managing complex and ambiguous decision scenarios.\\

Recent studies have increasingly focused on the development of aggregation, similarity, and correlation measures within the Fermatean fuzzy framework. Garg et al. \cite{gargetal} introduced a family of aggregation operators constructed using Yager’s t-norm and t-conorm to aggregate Fermatean fuzzy information in decision-making contexts. An interval-valued Fermatean fuzzy (IVFF)-based hybrid MCDM approach was proposed in \cite{kirisciASC} for risk assessment in autonomous vehicle systems.\\

Kirisci \cite{kirisci0} developed novel correlation coefficients for Fermatean hesitant fuzzy and interval-valued Fermatean hesitant fuzzy elements by employing least common multiple expansions. Furthermore, a three-way approach for computing correlation coefficients among FFSs based on variance and covariance concepts was presented in \cite{kirisci11}. In \cite{kirisciNew}, new correlation and weighted correlation coefficients were proposed to analyze relationships between FFSs.\\

Additional contributions include the introduction of Fermatean hesitant fuzzy sets and corresponding aggregation operators by Kirisci \cite{kirisciIVFFL}, as well as an interval-valued Fermatean fuzzy linguistic kernel principal component analysis model \cite{kirisciKFCA}. The concepts of Fermatean fuzzy soft sets and their fundamental properties were discussed in \cite{kirisci1}, alongside definitions of Fermatean fuzzy soft entropy and standard distance measures such as Hamming and Euclidean distances. Prioritized aggregation operators under Fermatean fuzzy environments were developed by Riaz et al. \cite{riazetal}, while group decision-making models based on incomplete Fermatean fuzzy preference relations were proposed in \cite{kirisci22}. Earlier foundational work on interval-valued intuitionistic fuzzy sets \cite{AtanGargov} and interval-valued Pythagorean fuzzy sets laid the groundwork for the later introduction of interval-valued Fermatean fuzzy sets by Jeevaraj \cite{Jeevaraj}.

	\section{Preliminaries}\label{chap:1}

%\subsection{Interval-Valued Fermatean Fuzzy Sets}
\begin{Definition}\cite{SenYager}
	Let $E$ be the universal set. The FFS is defined as the set $\mathcal{F}=\{(x, \zeta_{\mathcal{F}}(x), \eta_{\mathcal{F}}(x)): x \in E\}$, where with $0 \leq \zeta_{\mathcal{F}}^{3}+\eta_{\mathcal{F}}^{3} \leq 1$ and $\zeta_{\mathcal{F}}, \eta_{\mathcal{F}} \in [0,1]$. The hesitation degree has been shown with $h_{\mathcal{F}}=(1-\zeta_{\mathcal{F}}^{3}+\zeta_{\mathcal{F}}^{3} )^{1/3}$.
\end{Definition}

Consider the three FFSs $\mathcal{F}=(\zeta_{\mathcal{F}}, \eta_{\mathcal{F}})$, $\mathcal{F}_{1}=(\zeta_{\mathcal{F_{1}}}, \eta_{\mathcal{F_{1}}})$, $\mathcal{F}_{2}=(\zeta_{\mathcal{F_{2}}}, \eta_{\mathcal{F_{2}}})$. Then,

\begin{itemize}
	\item [(i)] $\mathcal{F}_{1} \cap \mathcal{F}_{2}=\left(\min(\zeta_{\mathcal{F}_{1}}, \zeta_{\mathcal{F}_{2}}), \max(\eta_{\mathcal{F}_{1}}, \eta_{\mathcal{F}_{2}})\right)$,
	\item [(ii)] $\mathcal{F}_{1} \cup \mathcal{F}_{2}=\left(\max(\zeta_{\mathcal{F}_{1}}, \zeta_{\mathcal{F}_{2}}), \min(\eta_{\mathcal{F}_{1}}, \eta_{\mathcal{F}_{2}})\right)$,
	\item [(iii)] $\mathcal{F}^{c}=(\eta_{\mathcal{F}}, \zeta_{\mathcal{F}} )$,
	\item [(iv)]  $\mathcal{F}_{1} \oplus \mathcal{F}_{2}=\left( \sqrt[3]{\zeta_{\mathcal{F}_{1}}^{3} + \zeta_{\mathcal{F}_{2}}^{3} - \zeta_{\mathcal{F}_{1}}^{3}\zeta_{\mathcal{F}_{2}}^{3}}, \eta_{\mathcal{F}_{1}}\eta_{\mathcal{F}_{2}} \right)$,
\item [(v)]  $\mathcal{F}_{1} \otimes \mathcal{F}_{2}=\left(\zeta_{\mathcal{F}_{1}}\zeta_{\mathcal{F}_{21}}, \sqrt[3]{\eta_{\mathcal{F}_{1}}^{3} + \eta_{\mathcal{F}_{2}}^{3} - \eta_{\mathcal{F}_{1}}^{3}\eta_{\mathcal{F}_{2}}^{3}}\right)$,
    \item [(v)] $\alpha \mathcal{F}=\left( \sqrt[3]{1-(1-\zeta_{\mathcal{F}}^{3})^{\alpha}},  (\eta_{\mathcal{F}})^{\alpha} \right)$,
    \item [(vi)] $\mathcal{F}^{\alpha}=\left((\zeta_{\mathcal{F}}^{3})^{\alpha},  \sqrt[3]{1-(1-\eta_{\mathcal{F}}^{3})^{\alpha}},  \right)$.
\end{itemize}

\begin{Definition}\cite{ranMis1}
	For the FFS $\mathcal{F}$,
	\begin{eqnarray}\label{scr}
		SC(\mathcal{F})&=&\frac{1}{2}\left( \zeta_{\mathcal{F}}^{3}(a) -  \eta_{\mathcal{F}}^{3}(a) \right) \\ \label{acc}
		AC(\mathcal{F})&=&\frac{1}{2}\left( \zeta_{\mathcal{F}}^{3}(a) +  \eta_{\mathcal{F}}^{3}(a) \right) \\ \label{Nscr}
		\overline{SC}(\mathcal{F})&=&\frac{1}{2}\left(SC(\mathcal{F})+1 \right)
	\end{eqnarray}
	are called score, accuracy, and normalized score functions, respectively, where $ SC(\mathcal{F}) \in [-1, 1] $,  $ AC(\mathcal{F}) \in [0, 1] $, and $ \overline{SC}(\mathcal{F}) \in [0, 1] $.
\end{Definition}

\begin{Definition}
For $i=1,2, \cdots, n$, consider a set of FFSs $\mathcal{F}_{i}=(\zeta_{\mathcal{F_{i}}}, \eta_{\mathcal{F_{i}}})$. The equation
\begin{eqnarray}\label{FFWA}
	FFWA(\mathcal{F}_{1}, \cdots, \mathcal{F}_{n})&=& \left( \sum_{i=1}^{n} \omega_{i}\zeta_{\mathcal{F}_{i}},  \sum_{i=1}^{n} \omega_{i}\eta_{\mathcal{F}_{i}} \right)\\ \nonumber
	&=&  \Bigg(
	\Bigg(1-\prod_{i=1}^{n} (1-\left(\zeta_{i} \right)^{3})^{\omega_{i}} \Bigg)^{1/3}, 	\prod_{i=1}^{n} \left(\eta_{i} \right)^{\omega_{i}}   \Bigg)
\end{eqnarray}
is called the FF-weighted average (FFWA) operator, and
the equation
\begin{eqnarray}\label{FFWG}
	FFWG(\mathcal{F}_{1}, \cdots, \mathcal{F}_{n})&=& \left( \prod_{i=1}^{n} \zeta_{\mathcal{F}_{i}}^{\omega_{i}},  \prod_{i=1}^{n} \eta_{\mathcal{F}_{i}}^{\omega_{i}} \right)\\ \nonumber
	&=&  \Bigg(  \prod_{i=1}^{n} \left(\zeta_{i} \right)^{\omega_{i}},
\Bigg(1-\prod_{i=1}^{n} (1-\left(\eta_{i} \right)^{3})^{\omega_{i}} \Bigg)^{1/3}     \Bigg)
\end{eqnarray}
is called the FF-weighted geometric (FFWG) operator,
where $\omega_{i}$ is a weight vector of $\mathcal{F}_{i}$ with $\sum_{i=1}^{n}\omega_{i}=1$.
\end{Definition}

	\subsection{PIPRECIA Technique}
	The Pivot Pairwise Relative Criteria Importance Assessment (PIPRECIA) method is a multi-criteria decision-making approach developed to derive the relative weights of evaluation criteria in decision problems. The method was introduced in 2015 by Pamucar, Lazic, and Bozanic as an alternative weighting technique tailored to decision environments characterized by subjective judgments and uncertainty. PIPRECIA relies on a systematic procedure in which experts assess the relative importance of criteria through sequential comparisons rather than exhaustive pairwise evaluations.\\
	
	One of the main advantages of PIPRECIA lies in its simplicity and practicality, which make it particularly suitable for situations where expert opinions are qualitative or imprecise. By reducing the number of required comparisons, the method minimizes cognitive burden and decreases the likelihood of inconsistency in expert assessments. PIPRECIA was specifically designed to overcome several shortcomings of traditional weighting methods, such as the complexity of the Analytic Hierarchy Process (AHP) and the sensitivity of the Step-wise Weight Assessment Ratio Analysis (SWARA) method to ordering bias.\\
	
	Although PIPRECIA, like AHP, is grounded in expert judgment, it avoids the use of extensive pairwise comparison matrices by employing a pivot-based relative evaluation mechanism. Compared to SWARA, the introduction of a pivot criterion enhances consistency and mitigates potential bias, leading to more reliable criterion weight estimation in subjective decision-making contexts.

\subsection{MARCOS Techniques}

The Measurement Alternatives and Ranking according to COmpromise Solution (MARCOS) method is a multi-criteria decision-making technique designed for evaluating and ranking alternatives under multiple, and often conflicting, criteria. Introduced in 2020, MARCOS extends and refines earlier compromise-based MCDM methods such as TOPSIS and VIKOR. The core principle of the method is the simultaneous consideration of ideal and anti-ideal reference points, representing the best and worst possible performance levels, respectively.\\

In the MARCOS framework, each alternative is assessed relative to both reference solutions, and a utility function is computed to quantify its overall performance. This utility-based evaluation enables a clear interpretation of how close each alternative is to the ideal solution while also accounting for its distance from the worst-case scenario. Consequently, the final ranking reflects not only proximity to the ideal alternative but also the relative influence of the least favorable outcomes.\\

A notable strength of the MARCOS method is its flexibility, as it can accommodate both qualitative and quantitative criteria across a wide range of decision-making problems. By integrating the concepts of reference solutions and compromise ranking into a unified framework, MARCOS provides a more comprehensive and robust ranking mechanism compared to conventional MCDM approaches.

\section{New Entropy Measure}

A new entropy metric is provided in this subsection. Additionally, a flexible cosinus function has been modified to provide more nonlinear responses to entropy in this innovation. The current trends in creating tools tailored to address nonlinearity, such as in MCDM, benefit from such solutions. \\

\begin{Definition}
	The mapping $\mathcal{E}: FFS(E) \rightarrow [0,1]$ is called a FF-entropy if the following conditions hold:
\begin{itemize}
	\item [E1.] $0 \leq \mathcal{E}(\mathcal{F}) \leq 1$.
	\item [E2.] $\mathcal{E}(\mathcal{F})=0$ iff $\mathcal{F}$ is a crisp set.
	\item [E3.] $\mathcal{E}(\mathcal{F})=1$ iff $\zeta_{\mathcal{F}}(a)=\eta_{\mathcal{F}}(a)$.
	\item [E4.] $\mathcal{E}(\mathcal{F}) \ leq \mathcal{E}(\mathcal{G})$ iff \\
	- $\mathcal{F} \subseteq \mathcal{G}$, if \\
	- $\mathcal{G} \subseteq \mathcal{F}$, if
	\item [E5.] $\mathcal{E}(\mathcal{F}^{c})=\mathcal{E}(\mathcal{F})$ for each $\mathcal{E}(\mathcal{F}) \in FFS(E)$.
\end{itemize}

\end{Definition}

\begin{Definition}
	Consider the FFS $\mathcal{F}=\{(a, \zeta_{\mathcal{F}}(a), \eta_{\mathcal{F}}(a)): a \in E\}$. The FF-entropy of $\mathcal{F}$ is defined as:
\begin{eqnarray}\label{newentropy}
	\mathcal{E}(\mathcal{F})=\frac{1}{n} \sum_{i=1}^{n}\left[\left\{ \sqrt{2} \cos\left(\pi\left(\frac{ \zeta_{\mathcal{F}}^{3}(a_{i}) - \eta_{\mathcal{F}}^{3}(a_{i})}{4} \right)  \right)-1	 \right\}\frac{1}{\sqrt{2}-1}	\right]
\end{eqnarray}
	\end{Definition}

\begin{Theorem}
 $\mathcal{E}(\mathcal{F})$ is a FF-entropy.
\end{Theorem}

\begin{proof}
	E1. We will show $0 \leq \mathcal{E}(\mathcal{F}) \leq 1$. It is known to be $0 \leq \zeta_{\mathcal{F}}^{3}(a),  \eta_{\mathcal{F}}^{3}(a) \leq 1$ and $0 \leq \zeta_{\mathcal{F}}^{3}(a)+  \eta_{\mathcal{F}}^{3}(a) \leq 1$. $\pi |\zeta_{\mathcal{F}}^{3}(a) -  \eta_{\mathcal{F}}^{3}(a)| \leq \pi$. Therefore, $0 \leq \mathcal{E}(\mathcal{F}) \leq 1$.\\
	
	If $\mathcal{E}(\mathcal{F})=0$, then $\pi |\zeta_{\mathcal{F}}^{3}(a) -  \eta_{\mathcal{F}}^{3}(a)|=\pi$. Thus, $\zeta_{\mathcal{F}}(a)=0$, $\eta_{\mathcal{F}}(a)=1$, $h_{\mathcal{F}}(a)=0$ or $\zeta_{\mathcal{F}}(a)=1$, $\eta_{\mathcal{F}}(a)=0$, $h_{\mathcal{F}}(a)=0$. This shows us that $\mathcal{F}$ is a crisp set. Conversely, if $\mathcal{F}$ is a crisp set, $\mathcal{E}(\mathcal{F})=0$.\\

	If $\mathcal{E}(\mathcal{F})=1$, then $\pi |\zeta_{\mathcal{F}}^{3}(a) -  \eta_{\mathcal{F}}^{3}(a)|=0$. Hence, $\zeta_{\mathcal{F}}(a)-\eta_{\mathcal{F}}(a)=0$, $\Rightarrow$ $\zeta_{\mathcal{F}}(a)=\eta_{\mathcal{F}}(a)$. Conversely, if $\zeta_{\mathcal{F}}(a)=\eta_{\mathcal{F}}(a)$, $\mathcal{E}(\mathcal{F})=1$.\\
	
	Write the function $U(x,y)=\left[\left\{ \sqrt{2} \cos\left(\pi\left(\frac{ x^{3} - y^{3}}{4} \right)  \right)-1	 \right\}\frac{1}{\sqrt{2}-1}	\right]$. To perform this operation, two situations will be taken into account:  $x \leq y$ and $x \geq y$.\\
	
	For $x \leq y$, $V(x,y)=\left[\left\{ \sqrt{2} \cos\left(\pi\left(\frac{ x^{3} - y^{3}}{4} \right)  \right)-1	 \right\}\frac{1}{\sqrt{2}-1}	\right]$. It must be proven that $V(x,y)$ decreases with $x$ and increases with $y$. The partial derivatives for $x$ and $y$ are as:
	
	\begin{eqnarray*}
		\frac{\partial V(x,y)}{\partial x} = \frac{-3\sqrt{2}\pi x^{2}}{4(\sqrt{2}-1)} \sin \left( \pi\left(\frac{ x^{3} - y^{3}}{4}\right) \right) \leq 0, \\
\frac{\partial V(x,y)}{\partial y} = \frac{3\sqrt{2}\pi y^{2}}{4(\sqrt{2}-1)} \sin \left( \pi\left(\frac{ x^{3} - y^{3}}{4}\right) \right)  \geq 0.
	\end{eqnarray*}

From this results, 	the conditions $\zeta_{\mathcal{G}}(a) \leq \eta_{\mathcal{G}}(a)$ and $\zeta_{\mathcal{F}}(a) \leq \zeta_{\mathcal{G}}(a)$, $\eta_{\mathcal{F}}(a) \geq \eta_{\mathcal{G}}(a)$ hold. Hence, $U(\zeta_{\mathcal{F}}(a), \eta_{\mathcal{F}}(a)) \leq U(\zeta_{\mathcal{G}}(a), \eta_{\mathcal{G}}(a))$.\\

For $x \geq y$, $Z(x,y)=\left[\left\{ \sqrt{2} \cos\left(\pi\left(\frac{ y^{3} - x^{3}}{4} \right)  \right)-1	 \right\}\frac{1}{\sqrt{2}-1}	\right]$. It must be proven that $V(x,y)$ increases with $x$ and decreases with $y$. Then,

	\begin{eqnarray*}
	\frac{\partial Z(x,y)}{\partial x} = \frac{3\sqrt{2}\pi y^{2}}{4(\sqrt{2}-1)}\sin \left(\pi\left(\frac{ y^{3} - x^{3}}{4}\right) \right) \geq 0,\\
	\frac{\partial Z(x,y)}{\partial y} = \frac{-3\sqrt{2}\pi x^{2}}{4(\sqrt{2}-1)} \sin \left(\pi\left(\frac{ y^{3} - x^{3}}{4}\right) \right) \leq 0.
\end{eqnarray*}
That is,
	the conditions $\zeta_{\mathcal{G}}(a) \geq \eta_{\mathcal{G}}(a)$ and $\zeta_{\mathcal{F}}(a) \geq \zeta_{\mathcal{G}}(a)$, $\eta_{\mathcal{F}}(a) \leq \eta_{\mathcal{G}}(a)$ hold.  Hence, $U(\zeta_{\mathcal{F}}(a), \eta_{\mathcal{F}}(a)) \leq U(\zeta_{\mathcal{G}}(a), \eta_{\mathcal{G}}(a))$. Therefore, $\mathcal{F} \leq \mathcal{G} \Rightarrow \mathcal{E}(\mathcal{F}) \leq \mathcal{E}(\mathcal{G})$.\\

Finally,
\begin{eqnarray*}
\mathcal{E}(\mathcal{F})=\frac{1}{n} \sum_{i=1}^{n}\left[\left\{ \sqrt{2} \cos\left(\pi\left(\frac{ \zeta_{\mathcal{F}}^{3}(a_{i}) - \eta_{\mathcal{F}}^{3}(a_{i})}{4} \right)  \right)-1	 \right\}\frac{1}{\sqrt{2}-1}	\right] = \\
\frac{1}{n} \sum_{i=1}^{n}\left[\left\{ \sqrt{2} \cos\left(\pi\left(\frac{ \eta_{\mathcal{F}}^{3}(a_{i}) - \zeta_{\mathcal{F}}^{3}(a_{i})}{4} \right)  \right)-1	 \right\}\frac{1}{\sqrt{2}-1}	\right] =\mathcal{E}(\mathcal{F}^{c}).
\end{eqnarray*}
\end{proof}

\section{Proposed Method}
By combining the cutting-edge FF-entropy, FF-PIPRECIA, and FF-MARCOS tools, this study suggests an integrated method for handling MCDM problems under FFSs. Using pseudo representations, the fundamental ideas of FFSs and the related MCDM approaches have been explained separately. \\

Let the sets of alternatives and criteria be given as follows, respectively: $A=\{A_{1}, A_{2}, \cdots, A_{m}\}$ and $K=\{K_{1}, K_{2}, \cdots, K_{n}\}$. A senior member assembled a group of $l$ experts ($S_{1}, S_{2}, \cdots, S_{l}$). \\

\textbf{LEVEL A:} Obtain the decision matrices based on experts. \\

Step 1:Identify the alternative, criteria and experts sets:\\

Step 2: Get the linguistic variables of experts and convert linguistic variables to FF-numbers(Table \ref{table000}). Hence, the expression for findnig the weight is given by

\begin{eqnarray}\label{weightDMR}
	\omega^{s}=\frac{1}{2}\left( \frac{\frac{1}{2}(\zeta_{F_{s}}^{3}-\eta_{F_{s}}^{3})+1}{\sum_{s=1}^{t}(\frac{1}{2}(\zeta_{F_{s}}^{3}-\eta_{F_{s}}^{3})+1)}+ \frac{s- \overline{SC}(s)+1}{\sum_{s=1}^{t}(s- \overline{SC}(s)+1)} \right),
\end{eqnarray}
where $\omega_{s} \geq 1$ and $\sum_{s=1}^{t}\omega_{s}=1$.

\begin{table}[htb!]
	\centering
	\caption{Scale Values according to FFN for experts}\label{table000}
	\begin{tabular}{ll}
		\hline
		Linguistic Terms & FFNs\\ \hline
		Absolutely Important (AI) & (1.00, 0.00)\\
		Very  Important (VI) & (0.85, 0.25)\\
		Important (I) & (0.70, 0.40)\\
		Medium (M) & (0.50, 0.50)\\
		Low (L) & (0.40, 0.70)\\
		Very Low (VL) & (0.25, 0.85)\\
		Unimportant (U) & (0.00, 1.00)\\
		\hline
	\end{tabular}
\end{table}

Step 3: Create the aggregated DEMA based on FF-numbers by Equation \ref{FFWA} (or \ref{FFWG}) with DEs weights.\\

\textbf{LEVEL B:} Find the weights.\\

Step 4: Compute the entropy on FFSs of the criterion by Equation \ref{newentropy}.\\

Step 5: Calculate the weight of the criterion by Equation \ref{weight2}.

\begin{eqnarray} \label{weight2}
	\omega_{j}^{O}=\frac{1-\mathcal{E}_{j}}{\sum_{j=1}^{n}(1-\mathcal{E}_{j})}
\end{eqnarray}
where $\mathcal{E}_{j}$ represented the entropy.

\textbf{LEVEL C:} PIPRECIA method.\\

%%%%PIPRECIA%%%%%
For subjective weights, we employ the FF-PIPRECIA model. This approach begins by considering the relevant evaluation criteria and applying the FF-score function to rate their significance.\\

Step 6: DMs evaluate the criteria to ascertain their relative relevance, starting with the second criterion:

\begin{eqnarray}\label{rating}
	s_{j}= \left\{ \begin{array}{cll}
		1+[\overline{SC}(U_{j})-\overline{SC}(U_{j-1})]&, & \quad \textit{if} \quad U_{j} > U_{j-1},\\
		1&, & \quad \textit{if} \quad U_{j} = U_{j-1},\\
		1-[\overline{SC}(U_{j-1})-\overline{SC}(U_{j})]&, & \quad \textit{if} \quad U_{j} < U_{j-1},
	\end{array} \right.
\end{eqnarray}
where $U_{j}$ and $U_{j-1}$ symbolize the significance rating of the criterion $j$th and $(j-1)$th criterion, respectively.\\

Step 7: Based on relative significance, compute the coefficient by
\begin{eqnarray}\label{coef}
	K_{j}= \left\{ \begin{array}{cll}
		1&, & \quad \textit{if} \quad j=1,\\
		2-s_{j}&, & \quad \textit{if} \quad j>1.
	\end{array} \right.
\end{eqnarray}

Step 8: Determine the initial weight by

\begin{eqnarray}\label{iniwe}
	Q_{j}= \left\{ \begin{array}{cll}
		1&, & \quad \textit{if} \quad j=1,\\
		\frac{Q_{j-1}}{K_{j}}&, & \quad \textit{if} \quad j>1.
	\end{array} \right.
\end{eqnarray}

Step 9: Obtain the subjective weight of $j$th criterion by

\begin{eqnarray}\label{subweight}
	\omega_{j}^{s}=\frac{Q_{j}}{\sum_{j=1}^{n}Q_{j}}, \quad \forall j.
\end{eqnarray}

Step 10: An integrated weight-determining model is presented as
\begin{eqnarray}\label{integrated}
	\omega_{j}= \alpha 	\omega_{j}^{O} + (1-\alpha)	\omega_{j}^{s}
\end{eqnarray}
to get the benefits of objective and subjective weighting models, where $j=1,2,\cdots,n$ and $\alpha \in [0,1]$ represents the strategic coefficient to assess the changes of criterion weights.\\

%%%%%MARCOS%%%%%

\textbf{LEVEL D:} MARCOS method.\\

The relationship between options and the ideal and negative-ideal alternatives on FF data is explained by the FF-MARCOS method.\\ 

Step 11: The combined FF-decision matrix was normalized by

\begin{eqnarray}\label{normalized}
	N_{ij}= \left\{ \begin{array}{cll}
		(\zeta_{ij}, \zeta_{ij})&, & \quad \textit{for benefit-type criteria},\\
		(\eta_{ij}, \eta_{ij})&, & \quad \textit{for cost-type criteria}.
	\end{array} \right.
\end{eqnarray}

Step 12: Computing positive ideal solutions and negative ideal solutions by
\begin{eqnarray*}
	N_{j}^{+}=\max_{i}N_{ij} \quad \textit{and} \quad 	N_{j}^{-}=\min_{i}N_{ij}.
\end{eqnarray*}

Step 13: Find the weighted normalized FF-decision matrix using the Equation \ref{WNFF}.\\

\begin{eqnarray}\label{WNFF}
  P_{ij}^{\omega} = N_{ij} \times \omega_{j}=\left(\zeta_{N_{ij}}^{\omega_{i}}, \sqrt[3]{1 - (1-\eta_{N_{ij}}^{3})^{\omega_{i}}}\right)
\end{eqnarray}

Step 14: Obtain the weighted sum of each option using score function by

\begin{eqnarray}\label{weightedsum}
	S_{i}=\sum_{j=1}^{n} \overline(SC)(N_{ij}),
\end{eqnarray}
where $\overline(SC)(N_{ij})$ signifies the score values of each element of the weighted normalized IVFF-decision matrix.\\

Step 15: Computing the utility degree with the following equations:
\begin{eqnarray}\label{utility}
	U_{i}^{-}=\frac{S_{i}}{S_{PIS}} \quad \textit{and} \quad U_{i}^{+}=\frac{S_{i}}{S_{NIS}},
\end{eqnarray}
where $S_{PIS}$ and $S_{NIS}$ signify the sum of score degrees of weighted values $N_{j}^{+}$ and $N_{j}^{-}$.\\

Step 16: For $f(U_{i}^{+})=\frac{U_{i}^{-}}{U_{i}^{-}+U_{i}^{+}}$ and $f(U_{i}^{-})=\frac{U_{i}^{+}}{U_{i}^{-}+U_{i}^{+}}$, the final values of utility functions by

\begin{eqnarray}\label{finval}
	f(U_{i})=\frac{U_{i}^{+}+U_{i}^{-}}{1+\frac{1-f(U_{i}^{+})}{f(U_{i}^{+})}+\frac{1-f(U_{i}^{-})}{f(U_{i}^{-})}}.
\end{eqnarray}

Step 17: Use Equation \ref{finval} to rank the options. The maximum values derived from Equation \ref{finval} are the right option.

%\begin{figure}
%	\centering
%	\includegraphics[width=13cm]{yeni_diag.jpg}
%	\caption{Flowchart of method}\label{flwchrt}
%\end{figure}
%----------------------

%-----------------------------

\section{Analysis of  Energy Poverty in T\"{u}rkiye}

\textbf{Step 1: }
This study will use an integrated MCDM technique to regionally rank energy poverty in 7 regions of T\"{u}rkiye (Marmara, Aegean, Mediterranean, Central Anatolia, Eastern Anatolia, Southeastern Anatolia, Black Sea) based on the following factors:\\

Factors: A-Income level; B-Energy price; C-Energy efficiency; D-Building efficiency; E-Climate; F-Urbanization.\\

The template for the questions asked of the experts is as follows:\\

"What is the significance of the $[F]$'s impact on energy poverty in the $[R]$ Region?"\\

The content of expressions $[F]$ and $[R]$ in this question is as follows:\\
$[R]$ - $R_{1}$: Marmara, $R_{2}$: Ege, $R_{3}$: Karadeniz, $R_{4}$: Akdeniz, $R_{5}$: Doğu Anadolu, $R_{6}$: Güneydoğu Anadolu, $R_{7}$: İç Anadolu\\
$[F]$ -  $A$: Income level, $B$: Energy price, $C$: Energy efficiency, $D$: Building efficiency, $E$: Climate, $F$: Urbanization\\

The 42 questions asked for each region and each factor are arranged as follows: $A_{1}-A_{7}$,$\cdots$, $F_{1}-F_{7}$.\\

\textbf{Step 2:} $U_{1}=AI$, $U_{2}=VI$, and $U_{3}=VI$ are the language variables representing the experts' weights. Equation \ref{weightDMR} is used to determine the weights of the DEs, and $\omega_{s}=\{0.36, 0.32, 0.32\}$ is the result.\\

\textbf{Step 3:} Use Equation \ref{FFWA} to create the aggregated IVFF-DEMA using FF-numbers (Tables \ref{t1}, \ref{t2}). \\

\begin{table}[htb!]
	\centering
	%\tiny
	\caption{Linguistic Terms}\label{t1}
	\begin{tabular}{lllllll}
		\hline
			& A & B & C & D & E & F \\ \hline
	$R_{1}$	& (VI, I, I) & (AI, I, M) & (I, M, L) & (L, VI, VI) & (M, M, L) & (AI, L, M)\\
	$R_{2}$ & (I, I, VI)  & (VI, I, I) & (M, L, L) & (L, M, L) & (L, M, L) & (M, M, M)\\
	$R_{3}$ & (I, I, VI) & (I, I, I) & (M, M, M) & (L, L, L) & (I, L, M) & (M, M, M)\\
	$R_{4}$ & (M, I, VI) & (M, I, I) & (VL, M, M) & (VL, M, L) & (VL, L, L) & (M, L, M)\\
	$R_{5}$ & (VI, I, I) & (I, I, VI) & (I, M, M) & (I, M, L) & (VI, L, L) & (L, M, I)\\
	$R_{6}$ & (M, I, VI) & (M, I, I) & (L, M, M) & (L, L, M) & (L, I, I) & (L, M, M)\\
	$R_{7}$ & (I, I, I) & (VI, I, M) & (I, M, I) & (I, I, I) & (VI, M, M) & (L, M, L)\\
					\hline
	\end{tabular}
\end{table}

\begin{table}[htb!]
	\centering
	%\tiny
	\caption{Aggregated DEMA}\label{t2}
	\begin{tabular}{lllllll}
		\hline
		& A & B & C & D & E & F \\ \hline
		$R_{1}$ & (0.34, 0.92) & (0.00, 1.00) & (0.51, 0.82) & (0.36, 0.91) & (0.56, 0.78) & (0.00, 1.00)\\
		$R_{2}$ & (0.34, 0.91)  & (0.34, 0.92) & (0.62, 0.76) & (0.63, 0.76) & (0.63, 0.76) & (0.50, 0.80)\\
		$R_{3}$ & (0.91, 0.34) & (0.89, 0.40) & (0.80, 0.50) & (0.74, 0.70) & (0.82, 0.51) & (0.80, 0.50)\\
		$R_{4}$ & (0.90, 0.37) & (0.86, 0.43) & (0.75, 0.61) & (0.73, 0.67) & (0.70, 0.75) & (0.78, 0.56)\\
		$R_{5}$ & (0.92, 0.34) & (0.91, 0.34) & (0.84, 0.46) & (0.82, 0.51) & (0.86, 0.48) & (0.82, 0.53)\\
		$R_{6}$ & (0.90, 0.37) & (0.86, 0.43) & (0.78, 0.56) & (0.76, 0.63) & (0.85, 0.49) & (0.78, 0.56)\\
		$R_{7}$ & (0.89, 0.40) & (0.90, 0.36) & (0.86, 0.43) & (0.89, 0.40) & (0.88, 0.40) & (0.75, 0.63)\\
		\hline
	\end{tabular}
\end{table}

\textbf{Step 4:} Obtain the entropy.  $\mathcal{E}_{A}=8.26$, $\mathcal{E}_{B}=7.53$,   $\mathcal{E}_{C}=6.41$, $\mathcal{E}_{D}=7.24$, $\mathcal{E}_{E}=6.90$, and $\mathcal{E}_{F}=5.01$.\\

\textbf{Step 5:} Using the Equation \ref{weight2}. Then obtain the weights of all criteria: $\omega^{O}_{A}=0.204$, $\omega^{O}_{B}=0.184$, $\omega^{O}_{C}=0.158$,  $\omega^{O}_{D}=0.175$, $\omega^{O}_{E}=0.166$ and $\omega^{O}_{F}=0.113$.\\

%%%%%%%%%%%%%%%%%%%%%%%%%%%%%%%%%%%

The FF-PIPRECIA model will be used to determine the subjective weight of the criteria. The subjective weights were calculated using Equations \ref{rating}–\ref{iniwe}, and the results are displayed in Tables \ref{t4} and \ref{t5}. \\

\begin{table}[htb!]
	\centering
	\caption{Significance ratings of criteria}\label{t4}
	\begin{tabular}{lllllc}
		\hline
		& $S_{1}$ & $S_{2}$ & $S_{3}$ & aggregated values & crisp values\\ \hline
		A & VI & AI & VI &  (1.00, 0.00)  & 0.75\\
		B & VI & AI & I  &(1.00, 0.00)& 0.75\\
		C & L &  L & L &(0.74, 0.70)& 0.52\\
		D & I & VI & I &(0.91, 0.34)& 0.68\\
		E & M & M & I  &(0.83, 0.47)& 0.62\\
		F & VL & VL & VL  &(0.63, 0.85)& 0.41\\
		 \hline
	\end{tabular}
\end{table}

\begin{table}[htb!]
	\centering
	\caption{The weight of different alternative using FF-PIPRECIA}\label{t5}
	\begin{tabular}{llllll}
		\hline
		& Crisp degrees & $s_{j}$ & $\kappa_{j}$ & $Q_{j}$ & $\omega_{j}^{S}$\\ \hline
		A & 0.75  & -     & 1.000   & 1.000    & 0.206\\
		B & 0.75  & 1.000 & 1.000   & 1.000    & 0.206\\
		C & 0.52  & 0.770 & 1.230   & 0.813    & 0.167\\
		D & 0.68  & 0.840 & 1.160   & 0.701    & 0.144\\
		E & 0.62  & 0.940 & 1.060   & 0.661    & 0.136\\
		F & 0.41  & 1.790 & 1.121   & 0.690    & 0.142\\
 \hline
	\end{tabular}
\end{table}

Additionally, for $\alpha=0.5$ and $\omega_{j}=\{0.205, 0.195,  0.1625, 0.1595, 0.151, 0.1275\}$, the total weight of each criterion based on the FF-entropy and the FF-PIPRECIA model is assessed.\\

The relationship between options and the ideal and negative-ideal alternatives on FF data is explained by the FF-MARCOS method. The aggregated FF-decision matrix is converted into a normalized aggregated FF-decision matrix (Table \ref{t6}) since criteria A and B are non-beneficial, while the remaining criteria are advantageous. \\

\begin{table}[htb!]
	\centering
	%\tiny
	\caption{Normalized Aggregated DEMA}\label{t6}
	\begin{tabular}{lllllll}
		\hline
		& A & B & C & D & E & F \\ \hline
		$R_{1}$ & (0.92, 0.34) & (0.00, 1.00) & (0.51, 0.82) & (0.36, 0.91) & (0.56, 0.78) & (0.00, 1.00)\\
		$R_{2}$ & (0.91, 0.34)  & (0.34, 0.92) & (0.62, 0.76) & (0.63, 0.76) & (0.63, 0.76) & (0.50, 0.80)\\
		$R_{3}$ & (0.34, 0.91) & (0.40, 0.89) & (0.80, 0.50) & (0.74, 0.70) & (0.82, 0.51) & (0.80, 0.50)\\
		$R_{4}$ & (0.37, 0.90) & (0.43, 0.86) & (0.75, 0.61) & (0.73, 0.67) & (0.70, 0.75) & (0.78, 0.56)\\
		$R_{5}$ & (0.34, 0.92) & (0.34, 0.91) & (0.84, 0.46) & (0.82, 0.51) & (0.86, 0.48) & (0.82, 0.53)\\
		$R_{6}$ & (0.37, 0.90) & (0.43, 0.86) & (0.78, 0.56) & (0.76, 0.63) & (0.85, 0.49) & (0.78, 0.56)\\
		$R_{7}$ & (0.40, 0.89) & (0.36, 0.90) & (0.86, 0.43) & (0.89, 0.40) & (0.88, 0.40) & (0.75, 0.63)\\
		\hline
	\end{tabular}
\end{table}

\begin{table}[htb!]
	\centering
	%\tiny
	\caption{Weighted Normalized Aggregated DEMA}\label{t7}
	\begin{tabular}{lllllll}
		\hline
		& A & B & C & D & E & F \\ \hline
		$R_{1}$ & (0.64, 0.80) & (1.00, 0.00) & (0.50, 0.90) & (0.59, 0.85) & (0.45, 0.92) & (1.00, 0.00)\\
		$R_{2}$ & (0.63, 0.80)  & (0.63, 0.81) & (0.45, 0.93) & (0.44, 0.93) & (0.44, 0.93) & (0.44, 0.92)\\
		$R_{3}$ & (0.63, 0.80) & (0.60, 0.84) & (0.48, 0.90) & (0.43, 0.94) & (0.48, 0.90) & (0.44, 0.92)\\
		$R_{4}$ & (0.62, 0.82) & (0.56, 0.85) & (0.45, 0.92) & (0.42, 0.94) & (0.39, 0.96) & (0.43, 0.93)\\
		$R_{5}$ & (0.64, 0.80) & (0.62, 0.81) & (0.51, 0.88) & (0.50, 0.90) & (0.52, 0.90) & (0.46, 0.92)\\
		$R_{6}$ & (0.62, 0.82) & (0.56, 0.85) & (0.46, 0.91) & (0.44, 0.93) & (0.51, 0.90) & (0.43, 0.93)\\
		$R_{7}$ & (0.60, 0.83) & (0.61, 0.82) & (0.53, 0.87) & (0.56, 0.86) & (0.54, 0.87) & (0.41, 0.94)\\
		\hline
	\end{tabular}
\end{table}

The following are ideal solutions, both positive and negative:

\begin{eqnarray*}
	N_{j}^{+}&=&\Big( (1.00, 0.00), (0.63, 0.81), (0.63, 0.80), (0.62, 0.82), (0.64, 0.80), (0.62, 0.82), (0.61, 0.82)	\Big),\\
		N_{j}^{-}&=&\Big( (0.45, 0.92), (0.44, 0.92), (0.43, 0.94), (0.39, 0.96), (0.46, 0.92), 	(0.43, 0.93), (0.41, 0.94)\Big).
\end{eqnarray*}

Using $N_{j}^{+}, N_{j}^{-}$ and Table \ref{t7}, the IVFF-score degree of each option,  $N_{j}^{+}$,  and $N_{j}^{-}$  are determined and shown in Table \ref{t8}.

\begin{table}[htb!]
	\centering
	\caption{Score Values}\label{t8}
	\begin{tabular}{lllllllll}
		\hline
					& A & B & C & D & E& F& $N_{j}^{+}$ & $N_{j}^{-}$ \\ \hline
			$R_{1}$ 	& 0.44 & 0.75 & 0.35 & 0.40 & 0.33 & 0.75 & 0.75 & 0.33\\
			$R_{2}$ 	& 0.43 & 0.43 & 0. 32& 0.32 & 0.32 & 0.33 & 0.43 & 0.33  \\
			$R_{3}$ 	& 0.43 & 0. 43& 0.35 & 0.31 & 0.35 & 0.33 & 0.43 &0.31\\
			$R_{4}$ 	& 0.43 & 0.39 & 0. 33& 0.31 & 0.30 & 0.32 & 0.42  &0.30\\
			$R_{5}$ 	& 0.44 & 0.43 & 0.36 & 0.35 & 0.35 & 0.33 & 0.44  &0.33\\
			$R_{6}$	 & 0.42 & 0.40 & 0.34 & 0.32 & 0.35 & 0.32  & 0.42 &0.32\\
			$R_{7}$ 	& 0.41 & 0.42 & 0.42 & 0.38 & 0.37 & 0.31 & 0.42 &0.31\\
 \hline
	\end{tabular}
\end{table}

Using the Equation \ref{weightedsum}, $S_{1}=2.385$, $S_{2}=1.55$, $S_{3}=1.74$, $S_{4}=2.62$, $S_{5}=3.15$, $S_{6}=2.61$, $N_{j}^{+}=2.15$, $N_{j}^{-}=3.05$ are obtained. Therefore, from utility function values(Table \ref{t9}), the prioritization of options is $E>D>F>A>C>B$, and $E$ is the best choice.

\begin{table}[htb!]
	\centering
	\caption{Utility Degrees and Utility Functions}\label{t9}
	\begin{tabular}{lllll}
		\hline
				& $U_{i}^{-}$ & $U_{i}^{+}$ & $f(U_{i})$ & Ranking  \\ \hline
$R_{1}$ 	& 0.63 & 0.93 & 0.3172 & 1 \\
$R_{2}$ 	& 0.65 & 0.96 & 0.3170 & 2 \\
$R_{3}$ 	& 0.91 & 1.35 & 0.3165 & 3\\
$R_{4}$ 	& 0.70 & 1.04 & 0.3168 & 5\\	
$R_{5}$  	& 0.66 & 0.99 & 0.3154 & 7 \\
$R_{6}$ 	& 0.68 & 1.01 & 0.3162 & 6\\
$R_{7}$ 	& 0.62 & 0.92 & 0.3166 & 4\\

\hline
	\end{tabular}
\end{table}

An examination of energy poverty by region in T\"{u}rkiye yielded the following ranking: $R_{1} > R_{2} > R_{4} > R_{7} > R_{3} > R_{6} > R_{5}$. So, if we read this ranking from highest to lowest, we have the areas listed as follows: Marmara, Aegean, Mediterranean, Central Anatolia, Black Sea, Southeastern Anatolia, and Eastern Anatolia.

\section{Discussion}
The study's Application Section on energy poverty and the evaluation results for seven regions in T\"{u}rkiye can be explained as follows: The fact that Eastern and Southeastern T\"{u}rkiye are at the top and the Marmara and Aegean regions are at the bottom is an extremely robust result.\\

The reasons why the Eastern Anatolia region is identified as the region with the highest risk of energy poverty can be explained as follows: Income level is far below the Turkish average, the impact of energy prices is very high relative to income, building efficiency is generally poor, there is a widespread stock of old and uninsulated buildings, the climate is frigid, resulting in very high heating needs, urbanization is usually scattered, infrastructure is weak, and energy efficiency is low due to low technology and economic structure. Since all these criteria work negatively in the same direction, Eastern Anatolia is the most at-risk region within T\"{u}rkiye.\\

Examining the energy poverty in the Southeastern Anatolia region reveals that it is primarily caused by low income levels, high energy prices relative to income, low building efficiency, an extremely hot climate leading to energy shortages related to cooling, rapid but unplanned urbanization, and a limited energy-efficient industrial and technological infrastructure. In this region, energy poverty stemming from cooling rather than heating are more prominent.\\

The reasons for energy poverty in the Black Sea region include low income in middle-income and rural areas, the relatively high impact of energy prices, inadequate building insulation despite the humid climate, a long heating season, scattered urbanization, and high infrastructure costs due to the region's structure. Income is better than in the East, but climate and building factors increase energy poverty.\\

The Central Anatolia region is home to the country's capital. In terms of energy poverty, the region's income level is close to the Turkish average. Energy prices and building efficiency are moderate. However, the climate is frigid, and winters are harsh, resulting in high heating needs. While urbanization is high, it is relatively balanced. Energy poverty is moderate, with risks concentrated in specific sub-regions.\\

Income in the Mediterranean region is better than the Turkish average. Energy prices are relatively manageable. Building efficiency varies from place to place (there are significant differences between coastal and inland areas). Heating needs are low due to the mild climate. Urbanization has developed due to tourism and has a better infrastructure. Energy poverty is regional and seasonal.\\

Income levels are high in the Aegean region, while the impact of energy prices is relatively low. Building efficiency is above the Turkish average, and the climate is mild. Urbanization is generally orderly, and the infrastructure is strong. Energy poverty is limited and localized.\\

Marmara is the region with the lowest risk of energy poverty. Income levels in the Marmara region are the highest in T\"{u}rkiye. Energy prices impose a relatively low burden on income, while energy efficiency is the most developed in this area. Building efficiency is relatively high, and although urbanization is very dense, the infrastructure is strong. Poverty may exist in the region, but energy poverty is structurally limited.

\subsection{Comparison}
This analysis compares three objective weighting approaches: Proposed entropy (Cosinus-based Fermatean Fuzzy Entropy), Shannon-type Fermatean Fuzzy Entropy, and Linear (classical) Fermatean Fuzzy Entropy. For each entropy measure, objective criterion weights are recalculated, the FF-MARCOS method is applied as is, regional rankings are obtained, and a comparison is made with the proposed model.\\

The entropy of FSs measures the degree of fuzziness between FSs. The axiom construction for the entropy of FSs was initially presented by De Luca and Termini \cite{deluca} in relation to Shannon's probability entropy. A statistical metric frequently used to describe complex systems is the Shannon entropy. It can identify nonlinear features in model series, which helps provide a more trustworthy explanation of the nonlinear dynamics at various stages of analysis. This enhances understanding of the nature of complex systems characterized by complexity and nonequilibrium. The majority of complex systems are characterized by heterogeneous distributions of linkages, in addition to their complexity and nonequilibrium nature. Shannon utilized the concept of entropy, known as Shannon entropy, in information theory for data transmission in computer science. According to this theory, the division of the symbol logarithm in the alphabet by the entropy yields the mean value of the shortest options needed to code a message. \\

For probabilities $p_{1}, p_{2}, \cdots, p_{n}$, $X$ random vairable is deifned as
\begin{eqnarray*}
	H(p)=-\sum_{i=1}^{n}p_{i} \log_{2}p_{i}.
\end{eqnarray*}
$H(x)$ is defined as the entropy of $X$ random variable. \\

Let $X = \{x_{1}, x_{2},\cdots, x_{n}\}$ be a finite universe of discourses. Thus, for a FFS $\mathcal{F}$ in $X$, we propose a probability mass function $p=\{p_{1}, p_{2}, \cdots, p_{k}\}$ the following probability-type entropy for the FFS $\mathcal{F}$ with

\begin{eqnarray*}
	H(p)=-\frac{1}{n}\sum_{i=1}^{n} \left( \zeta_{\mathcal{F}}^{3}\log \zeta_{\mathcal{F}}^{3}(x_{i})+ \eta_{\mathcal{F}}^{3}\log \eta_{\mathcal{F}}^{3}(x_{i}) + h_{\mathcal{F}}^{3}\log h_{\mathcal{F}}^{3}(x_{i}) \right).
\end{eqnarray*}

For a FFS $\mathcal{F}$ in $X$, the FF-entropy measure $\mathcal{E}: FFS(X) \rightarrow [0,1]$ is defined as
\begin{eqnarray*}
\mathcal{E}=1-\frac{1}{4n}\sum_{i=1}^{n} \Bigg[2 \left| \zeta_{\mathcal{F}}^{3} - \eta_{\mathcal{F}}^{3} \right| + \left| \zeta_{\mathcal{F}}^{6} - \eta_{\mathcal{F}}^{6} \right| + \left|2 \left( \zeta_{\mathcal{F}}^{3} - \eta_{\mathcal{F}}^{3} \right) +  \left( h_{\mathcal{F}}^{6} - \eta_{\mathcal{F}}^{6} \right)\right| \Bigg]
\end{eqnarray*}

To investigate the robustness of the proposed framework with respect to the objective weighting mechanism, an entropy-based comparative analysis was conducted. The proposed cosinus-based FF-entropy was compared with two alternative formulations: the Shannon-type FF-entropy and the linear FF-entropy. For each entropy model, objective criterion weights were recalculated, and the FF-MARCOS method was reapplied. The resulting regional rankings were then compared using Kendall’s $\tau$ correlation coefficient. The results indicate perfect rank consistency ($\tau = 1.000$) across all entropy formulations, demonstrating that the proposed model is fully insensitive to the choice of entropy definition. This confirms that the robustness of the proposed framework is not driven by a specific entropy structure, but rather by its integrated decision-making architecture.\\

The ranking results of the compared entropy measures and the proposed entropy measure are shown in Table \ref{tableRank}.

	\begin{table}[htb!]
	\centering
	\tiny
	\caption{Ranking}\label{tableRank}
	\begin{tabular}{ll}
		\hline
		Entropy Model &	Ranking  \\  \hline
	Proposed Entropy		& $R_{1} > R_{2} > R_{4} > R_{7} > R_{3} > R_{6} > R_{5}$	\\
	Shannon type FF-entropy		&	$R_{1} > R_{2} > R_{4} > R_{7} > R_{3} > R_{6} > R_{5}$ \\
	Classical FF-entropy	&	$R_{1} > R_{2} > R_{4} > R_{7} > R_{3} > R_{6} > R_{5}$	\\ \hline
	\end{tabular}
\end{table}

	\subsection{Sensitivity: }
\textbf{$\alpha$ parameter analysis:}

A sensitivity study has been performed on the range of values for the $\alpha$ parameter. Regional rankings of T\"{u}rkiye's energy poverty issue have been established. We considered a range of $\alpha$ values in the sensitivity study. This analysis aims to ascertain the functionality of the recently developed framework. DEs can evaluate how sensitive the introduced model is to parameter changes by altering the $\alpha$ parameter. The best choice, the $R_{1}$ Marmara region, remains the same for all parameter values, as indicated by the sensitivity analysis results presented in Table \ref{table25} and Figure \ref{fig:sens}.  The range of values for the $\alpha$ parameter has been subjected to a sensitivity analysis. The situation of energy poverty in T\"{u}rkiye has been examined by region, and regional rankings have been obtained. In the sensitivity study, we considered a range of $\alpha$ values. The purpose of this examination is to determine how well the newly created framework functions. By modifying the $\alpha$ parameter, DEs can assess the sensitivity of the introduced model to changes in the parameters. The sensitivity analysis results in Table \ref{table25} and Figure \ref{fig:sens} indicate that the optimal option, $R_{1}$ Marmara region, remains unchanged for all parameter values. Because of this, the examination of studies on energy poverty in T\"{u}rkiye by geographic area is dependent on and sensitive to $\alpha$ values. Consequently, the proposed model is sufficiently stable over a variety of parameter values.  Nevertheless, it is also observed that the outcomes for $\alpha=0.5$ and $\alpha=0.0, 0.1, 0.2, 0.4, 0.7, 0.9, 1.0$ are identical. Figure  \ref{fig:sens}  illustrates how the endurance of the proposed framework is enhanced by adjusting the parameter degrees.

	\begin{table}[htb!]
		\centering
		\tiny
		\caption{Sensitivity Analysis}\label{table25}
		\begin{tabular}{lccccccccccc}
			\hline
& $\alpha=0.0$& $\alpha=0.1$ & $\alpha=0.2$ & $\alpha=0.3$ & $\alpha=0.4$ & $\alpha=0.5$ & $\alpha=0.6$ & $\alpha=0.7$ & $\alpha=0.8$ & $\alpha=0.9$ & $\alpha=1.0$\\ \hline
$R_{1}$&	0.560	&	0.562	&	0.568	&	0.570	&	0.575	&	\textbf{0.581}	&	0.583	&	0.585	&	0.590	&	0.594	&	0.596\\
$R_{2}$&	0.388	&	0.390	&	0.396	&	0.404	&	0.405	&   \textbf{0.388}	&	0.415	&	0.383	&	0.390	&	0.424	&	0.428\\
$R_{3}$&	0.243	&	0.239	&	0.235	&	0.227	&	0.225	&   \textbf{0.217}	&	0,215	&	0.210	&	0.207	&	0.204	&	0.202\\
$R_{4}$&	0.426	&	0.421	&	0.420	&	0.416	&	0.414	&	\textbf{0.412}	&	0.399	&	0.419	&	0.380	&	0.385	&	0.381\\ \hline
		\end{tabular}
	\end{table}

\begin{figure}[htb]
	\centering
	\includegraphics[width=15cm]{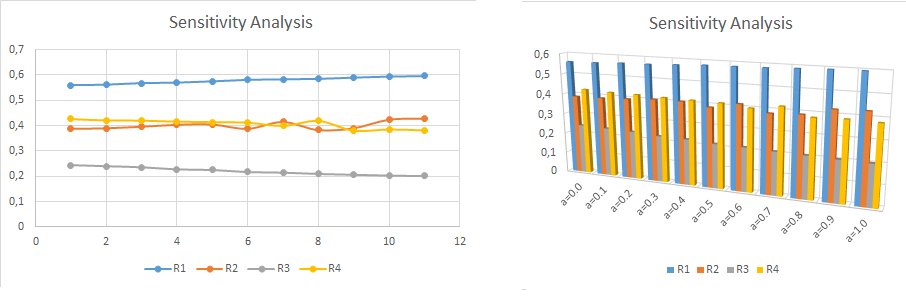}
	\caption{Sensitivity Analysis}\label{fig:sens}
\end{figure}

\textbf{Weight Perturbation Analysis:} A straightforward sequential parameter perturbation method is called Weight Perturbation \cite{jabri}. By calculating the change in error caused by a perturbation of a single parameter and adjusting that value accordingly, the approach sequentially updates the individual parameters. This method efficiently measures the gradient's components one after the other, requiring as many computational cycles as the system's parameters to fully understand the gradient. \\

Let's select  $\pm 10\%$ for Weight Perturbation Analysis. Initial criterion weights were $\omega={0.205,0.195,0.1625,0.1595,0.151,0.1275}$. Each criterion was considered individually. The relevant criterion weight was changed by $\pm 10\%$. The weights were renormalized. The regional ranking corresponding to the FF-MARCOS ranking was re-obtained. The agreement between the original ranking and the new ranking was measured using Kendall’s $\tau$ (Table \ref{tableKEN}).

	\begin{table}[htb!]
	\centering
	\tiny
	\caption{Kendall’s $\tau$ Results Table}\label{tableKEN}
	\begin{tabular}{lcc}
		\hline
Criterion &	Perturbation &	Kendall’s $\tau$ \\  \hline
Income			&	$+10\%$	&	0.429\\
Income			&	$-10\%$	&	0.429\\
Energy Price	&	$+10\%$	&	0.429\\
Energy Price	&	$-10\%$	&	0.429\\
Energy Efficiency & $+10\%$ & 0.429\\
Energy Efficiency  & $-10\%$ & 0.429\\
Building Efficiency & $+10\%$ & 0.429\\
Building Efficiency &$-10\%$& 0.429\\
Climate & $+10\%$& 0.429\\
Climate & $-10\%$& 0.429\\
Urbanization & $+10\%$& 0.429\\
Urbanization &$-10\%$ & 0.429\\ \hline
	\end{tabular}
\end{table}

To further examine the robustness of the proposed model, a weight perturbation analysis was conducted. Each criterion weight was independently increased and decreased by $\pm 10\%$, followed by re-normalization and re-evaluation of the FF-MARCOS ranking results. The consistency between the original and perturbed rankings was assessed using Kendall’s tau correlation coefficient. The results show that Kendall’s $\tau$ remains approximately 0.43 across all perturbation scenarios, indicating a moderate to high level of rank stability. Importantly, no drastic rank reversals were observed, demonstrating that the proposed framework is not overly sensitive to local changes in individual criterion weights and does not rely excessively on any single criterion.\\

This analysis allows us to clearly state the following: The model is robust to local weighting uncertainty. The results are not dragged by a single criterion. The FF-entropy + FF-PIPRECIA integration is stable.\\

\textbf{Entropy $\times$ Criterion Dominance Analysis:} Here, a methodological sensitivity analysis will be conducted as a robustness test. This analysis will answer the question: Does the relative dominance of the criteria change under different entropy definitions? In other words, the ranking may be stable, but is the effectiveness of each criterion sensitive to entropy selection?

For each entropy model, criterion dominance is defined using
\begin{eqnarray*}
	D_{j}^{(e)}=\frac{\omega_{j}^{(e)}}{\max_{i}\omega_{j}^{(e)}},
\end{eqnarray*}
where $\omega_{j}^{(e)}$ entropy model under criterion weighting $e$. Dominance scores are normalized to a range of [0, 1], with a score closer to 1 indicating a more dominant criterion.\\

	\begin{table}[htb!]
	\centering
	\tiny
	\caption{Entropy $\times$ Criterion Dominance}\label{tableDOM}
	\begin{tabular}{lcccccc}
		\hline
		Entropy Model &	Income & Energy Price & Energy Efficiency & Building Efficiency & Climate & Urbanization  \\  \hline
		Proposed Entropy		& 0.92 & 0.89 & 0.83 & 0.81 & 0.78 & 0.74	\\
		Shannon type FF-entropy		&	0.91 & 0.88 & 0.82 & 0.80 & 0.77 & 0.73 \\
		Classical FF-entropy	&	0.90 & 0.87 & 0.81 & 0.79 & 0.76 & 0.72	\\ \hline
	\end{tabular}
\end{table}

Table \ref{tableDOM} shows that the dominance ranking does not change according to entropy; income and energy price are dominant in all cases, and the sub-criteria decrease proportionally. Thus, it reveals a fairly consistent dominance structure across different entropy formulations. The relative influence of each criterion remains stable, and income and energy price consistently emerge as the most dominant factors under all entropy definitions.

\subsection{Implications}

\textbf{Theoretical Implications:} The results of this study reinforce the multidimensional nature of energy poverty. Thus, energy poverty studies theoretically strengthen the idea that energy poverty cannot be explained solely by income deficiency; we must consider together structural factors such as energy prices, building efficiency, climatic constraints, and spatial inequalities. Energy poverty is positioned as a subfield of the multidimensional welfare poverty literature, distinct from income-based poverty theories. Furthermore, this study will contribute to the deepening of demand-side energy poverty theory. Energy poverty analyses emphasize the role of demand-side constraints (inefficient housing, climate pressure, low technological adaptation) rather than supply-side constraints. Energy poverty is not a problem of “lack of access to energy,” but rather a problem of “inability to use energy sustainably.” Regional inequality theories can be enriched with an energy dimension. Regional energy poverty studies add a new energy-based dimension of inequality to the spatial economics and regional development literature. Under the same national energy policy, very different levels of energy prosperity can emerge in other regions. Ultimately, it will contribute to the fields of uncertainty and decision theory. Fuzzy sets, Bayesian models, prospect or consensus-based MCDM approaches provide theoretical contributions to the analysis of social problems that are difficult to measure and involve uncertainty, such as energy poverty. Energy poverty requires uncertainty-based decision theories that go beyond classical deterministic measurements.\\

The case of T\"{u}rkiye supports the idea that energy poverty is a "socio-technical-spatial phenomenon" and reveals the inadequacy of one-dimensional definitions (income-based poverty). Interregional differences (East–West, coastal–inland, urban–rural divide) indicate that energy poverty is spatially clustered. While national averages are generally used in the energy poverty literature, this study emphasizes the necessity of sub-national (regional) analyses. This strengthens the literature on regional energy justice and spatial energy poverty. Studies in T\"{u}rkiye indicate that energy poverty in low-income regions is not primarily due to high consumption, but rather to inefficient buildings and technologies. This finding supports approaches that view energy poverty as a problem of structural inefficiencies rather than "excess demand."\\

\textbf{Managerial Implications:} The study enables the design of targeted intervention strategies. Energy poverty analyses show public administrators and energy organizations that the "one-size-fits-all" approach is ineffective. The actionable insight from this situation is that building insulation should be prioritized in cold regions, while cooling efficiency and tariff adjustments should be prioritized in hot areas. Increased efficiency in resource allocation, along with energy poverty maps and rankings, enables the targeting of limited public resources to regions that will yield the highest marginal benefit. The actionable insight from this situation can be explained as follows: With the same budget, energy poverty can be reduced faster in regions where building efficiency investments are made. Examining social risk management for energy companies reveals that energy poverty poses collection risk, social acceptance risk, and regulatory pressure risk for energy distribution companies and retail suppliers. Flexible payment plans and social tariffs can be developed in low-income regions. Establishing performance monitoring and early warning mechanisms is crucial. Energy poverty indicators can be used as an early warning system. It can be predicted in advance in which regions energy price increases will lead to social vulnerability.\\

The study shows that energy poverty is not experienced uniformly across T\"{u}rkiye: heating-related poverty is observed in Eastern Anatolia, cooling and summer load-related poverty in Southeastern Anatolia, and long heating seasons and building inefficiencies in the Black Sea region. A "one-size-fits-all" approach to combating energy poverty proves unsuccessful. It has been observed that support based solely on income transfers provides short-term relief but does not permanently reduce energy poverty. Resources should be directed towards building renovation and efficiency investments instead of direct subsidies. Energy poverty is directly linked to the building stock, urban fabric, and local infrastructure. Municipalities and local governments should be positioned as implementing and targeting actors, rather than being independent of central policies.\\

\textbf{Policy Implications:} The integration of energy poverty into the social policy agenda will develop. Energy poverty findings indicate that energy policies should be integrated with social policies. Energy poverty is the responsibility not only of energy ministries but also of social policy institutions. Regional energy policies will be developed. Uniform energy policies at the national level may ignore regional differences. Heating subsidies in cold regions and cooling subsidies in the hot areas should be differentiated regionally. The social dimension of energy efficiency policies will be a key focus. Energy efficiency should be considered not only as an environmental tool but also as a tool for social justice. Building renovation programs for low-income households directly reduce energy poverty. Pricing and subsidy reforms will be implemented. Energy poverty analyses provide an opportunity to question how existing targeted subsidies are utilized and who benefits from them. Targeted and conditional support is more effective than general subsidies. Compliance with and work with climate policies will be ensured. Energy poverty is directly related to the concept of a just transition. Climate policies, such as carbon pricing, can generate social backlash if not designed with consideration for their impact on energy poverty.\\

In T\"{u}rkiye, energy poverty is often addressed in a fragmented manner, encompassing social assistance and energy subsidies. Energy poverty should be defined as a cross-sectoral area, addressing social policy, housing policy, and energy policy. The findings reveal that Eastern and Southeastern Anatolia are structurally disadvantaged. National energy poverty strategies should include regionally weighted funding and incentive mechanisms to address regional disparities. Energy prices should not be considered in isolation, but rather in conjunction with income, climate, and building efficiency. Regionally and climate-sensitive tariffs, as well as social tariff applications, are more effective in reducing energy poverty.\\

Actionable Insights can be summarized as follows: In the short term (1-2 years), energy poverty maps should be created, and seasonal energy subsidies should be provided through social tariffs and billing. In the medium term (3-5 years), insulation and building renovations should be initiated in low-income homes, and regional energy efficiency programs should be established with support for energy-efficient appliances. In the long term (5+ years), the integration of regional development and energy policies, a just energy transition framework, and the creation of institutional structures to monitor energy poverty can be defined.\\

Energy poverty analyses provide not only a diagnostic tool for identifying vulnerable regions and households, but also a strategic framework for designing equitable, efficient, and region-sensitive energy policies that align social welfare objectives with long-term energy transition goals. This study demonstrates that energy poverty in T\"{u}rkiye is not merely an income-based problem; it is a multidimensional policy area that requires consideration of regional, structural, and climatic factors together.

\subsection{Limitations}

\textit{Limitations of the Proposed Method:} \\

Expert-based subjectivity and judgment dependency: The proposed framework partially relies on expert judgments, particularly in the FF-PIPRECIA stage. Although Fermatean fuzzy sets are employed to mitigate subjectivity and vagueness, the results may still reflect experts’ cognitive biases, experience levels, and domain knowledge. Different expert panels could lead to variations in criteria weights and rankings. \\

Sensitivity to linguistic scale definition: The transformation of linguistic terms into Fermatean fuzzy numbers is based on predefined scale values. While these scales are consistent with the literature, alternative linguistic mappings or scale calibrations may influence entropy values, weighting outcomes, and final rankings.\\

Parameter dependency ($\alpha$ coefficient) :The integrated weighting mechanism depends on the strategic coefficient $\alpha$, which balances objective (entropy-based) and subjective (FF-PIPRECIA-based) weights. Although sensitivity analysis demonstrates robustness, the selection of $\alpha$ remains a decision-maker-driven process and may affect outcomes in certain contexts. \\

Computational and methodological complexity: Compared to classical MCDM methods, the proposed entropy–FF-PIPRECIA–FF-MARCOS framework is mathematically and computationally more complex. This may limit its direct applicability for practitioners without sufficient expertise in fuzzy decision-making or access to computational tools.\\

Static and deterministic evaluation structure: The proposed method operates within a static decision-making framework, assuming that criteria importance and expert evaluations remain constant during the analysis. In reality, energy systems and socio-economic conditions are dynamic, and temporal variability is not explicitly modeled.\\

Assumption of criteria independence: The framework implicitly assumes that evaluation criteria are independent. However, in energy poverty contexts, strong interdependencies may exist (e.g., income level and building efficiency, climate and energy consumption), which are not explicitly captured in the current formulation.\\

\textit{Limitations of the Application (Energy Poverty in T\"{u}rkiye):}\\

Regional aggregation bias: The analysis is conducted at the NUTS-1 (seven-region) level, which may mask significant intra-regional disparities. Energy poverty often varies substantially within regions, particularly between urban and rural areas or across income groups.\\

Lack of household-level empirical data: The application relies primarily on expert evaluations rather than household microdata (e.g., energy expenditures, consumption patterns, or self-reported poverty indicators). Consequently, the results reflect structural and expert-perceived vulnerability rather than directly observed household experiences.\\

Limited set of criteria: Although six key criteria are selected to represent energy poverty comprehensively, other potentially relevant factors—such as energy infrastructure reliability, renewable energy access, energy subsidy schemes, or behavioral aspects—are not included.\\

Context-specific generalizability: The findings are specific to T\"{u}rkiye’s socio-economic, climatic, and institutional context. The regional rankings and dominant drivers of energy poverty may not be directly transferable to countries with different energy markets, housing stocks, or welfare systems.\\

Temporal snapshot: The application reflects a single-period assessment and does not capture seasonal or long-term changes in energy poverty, such as those driven by energy price shocks, economic crises, or climate variability.\\

Policy implementation gap: While the study provides prioritization and strategic insights, it does not evaluate the feasibility, costs, or effectiveness of specific policy interventions. Thus, the translation of rankings into concrete policy actions requires additional economic and institutional analysis.\\

\textbf{These limitations do not undermine the validity of the proposed framework}; instead, they reflect the inherent complexity of energy poverty as a socio-technical phenomenon and highlight avenues for future methodological and empirical extensions.

\section{Conclusion}

This paper proposes an integrated decision-support framework based on FFSs to address the multifaceted and complicated nature of energy poverty. Recognizing that energy poverty cannot be adequately explained through income-based indicators alone, the study incorporates structural, climatic, and spatial dimensions into a unified analytical model. By explicitly accounting for uncertainty, vagueness, and subjectivity in expert evaluations, the proposed framework offers a more realistic and robust approach for analyzing energy poverty in real-world decision environments.\\

From a methodological perspective, the primary contribution of this research lies in the development of a novel Fermatean fuzzy entropy measure. The proposed entropy, constructed with a nonlinear cosinus-based structure, provides enhanced sensitivity in capturing the degree of uncertainty embedded in Fermatean fuzzy information. It is beneficial for complicated multi-criteria DM situations because, in contrast to current entropy measurements, it permits a more flexible and discriminative representation of fuzziness.  By integrating this entropy with the FF-PIPRECIA method for subjective weighting and the FF-MARCOS method for ranking alternatives, the study establishes a comprehensive and coherent Fermatean fuzzy MCDM framework that effectively combines objective data-driven insights with expert judgment.\\

The suggested model's practical applicability and analytical power are demonstrated by its empirical application to the regional evaluation of energy poverty in T\"{u}rkiye. The results reveal apparent geographical differences, with the Marmara and Aegean regions exhibiting the lowest levels of energy poverty, and Eastern and Southeastern Anatolia emerging as the most vulnerable. These findings demonstrate that energy poverty in T\"{u}rkiye is not homogeneous but instead spatially clustered and influenced by region-specific factors, including income levels, urbanization patterns, construction inefficiencies, and the harshness of the climate. A sensitivity study reveals that the ranking results remain essentially constant across various parameter settings, thereby confirming the robustness and stability of the proposed framework. \\

Beyond methodological advances, the study offers important theoretical, managerial, and policy implications. Theoretically, it reinforces the view of energy poverty as a socio-technical and spatial phenomenon, extending the literature on multidimensional welfare poverty and energy justice. Managerially, the results provide actionable insights for energy planners and utility companies by emphasizing the need for region-specific intervention strategies rather than uniform policy measures. From a policy perspective, the findings underscore the importance of integrating energy policy with housing, social policy, and regional development strategies to effectively mitigate energy poverty.\\

Despite its contributions, this study is subject to certain limitations. The analysis relies on expert-based evaluations and a predefined set of criteria, which, although appropriate for capturing complex qualitative judgments, may not fully reflect household-level heterogeneity. Additionally, the application focuses on regional-level analysis within a single country, which may limit the direct generalizability of the results to other national contexts with different institutional and energy system structures.\\

These limitations open several promising avenues for future research. First, future studies could extend the proposed framework by incorporating household-level microdata or combining expert judgments with empirical consumption and expenditure data. Second, dynamic and longitudinal analyses could be conducted to examine how energy poverty evolves over time under changing energy prices, climate conditions, and policy interventions. Third, the proposed Fermatean fuzzy entropy-based framework may be adapted and compared with other advanced uncertainty modeling approaches, such as Bayesian, rough set, or hybrid fuzzy–probabilistic models. Finally, applying the model to cross-country or multi-regional comparative studies would further validate its robustness and enhance its contribution to the global energy poverty and energy justice literature.\\

In conclusion, this study provides both a methodological advancement and a practical decision-support tool for the assessment of energy poverty under uncertainty. By offering a flexible, robust, and region-sensitive analytical framework, it contributes to the design of more equitable, efficient, and evidence-based energy policies aligned with long-term sustainability and just energy transition goals.

\section*{Ethics declarations}
\subsection*{Ethical approval}
 This article does not contain any studies with human participants or animals performed by any of the authors.

\subsection*{Funding Details}
 This research did not receive any specific grant from public, commercial, or not-for-profit funding agencies.

\subsection*{Conflict of interest}
 The authors declare that they have no known competing financial interests or personal relationships that could have appeared to influence the work reported in this paper.

\subsection*{Data Availability Statement}
 The manuscript has no associated data.

\subsection*{Authorship contributions} All authors equally contributed to the design and implementation of the research, the analysis of the results, and the writing of the manuscript.

%\section*{Author Contributions}
%This article was written with equal contributions from both authors. The final manuscript
%was read and approved by both authors.

%\section*{Conflict of Interests}
%The authors declare that they have no competing interests.


\begin{thebibliography}{99}

\bibitem{abbas} Abbas, Q., Hanif, I., Taghizadeh-Hesary, F., Iqbal, W.,  Iqbal, N.  Improving the energy and environmental efficiency for energy poverty reduction. Poverty reduction for inclusive sustainable growth in developing Asia, 2021;231--248.

\bibitem{alkez} Al Kez, D., Foley, A., Abdul, Z. K.,  Del Rio, D. F. Energy poverty prediction in the United Kingdom: A machine learning approach. Energy Policy, 2024;184:113909.

\bibitem{Atan} Atanassov, K. Intuitionistic fuzzy sets. \emph{Fuzzy Sets and Systems}. \textbf{1986}, \emph{20}, 87--96.

\bibitem{AtanGargov} K. Atanassov, G. Gargov, \emph{Interval valued intuitionistic fuzzy sets}, Fuzzy Sets and Systems, \textbf{31}(3) (1989), 343--349.

\bibitem{attri} Attri, S.D., Singh S., Dhar, A., Powar S., Multi-attribute sustainability assessment of wastewater treatment technologies using combined fuzzy multi-criteria decision-making techniques, Journal of Cleaner Production 357 (2022) 131849

\bibitem{R1} Aytekin A., Gorcun O.F., Ecer F., Pamucar D., Karamasa C. Evaluation of the pharmaceutical distribution and warehousing companies through an integrated Fermatean fuzzy entropy-WASPAS approach, Kybernetes, 52:5561--5592, 2022

\bibitem{ayyildiz} Ayyildiz, E., Erdogan M., Gul, M. A comprehensive risk assessment framework for occupational health and safety in pharmaceutical warehouses using Pythagorean fuzzy Bayesian networks, Engineering Applications of Artificial Intelligence 135 (2024) 108763


\bibitem{bela} Belaïd, F. Implications of poorly designed climate policy on energy poverty: Global reflections on the current surge in energy prices. Energy Research \& Social Science, 2022;92:102790.

\bibitem{boemi} Boemi, S. N.,  Papadopoulos, A. M.  Energy poverty and energy efficiency improvements: A longitudinal approach of the Hellenic households. Energy and Buildings, 2019;197:242--250.

\bibitem{board} Boardman, B.  Fixing fuel poverty: Challenges and solutions. Routledge. 2013, https://doi.org/10.4324/9781849774962

\bibitem{bour} Bouzarovski, S.,  Petrova, S.  A global perspective on domestic energy deprivation: Overcoming the energy poverty–fuel poverty binary. Energy Research \& Social Science, 2015;10:31--40. https://doi.org/10.1016/j.erss.2015.06.007

\bibitem{bouz} Bouzarovski, S.,  Simcock, N. Spatializing energy justice. Energy Policy, 2019;107:640--648. https://doi.org/10.1016/j.erss.2019.

\bibitem{casi} Casillas, C. E.,  Kammen, D. M.  The energy-poverty-climate nexus. Science, 2010;330(6008):1181--1182.

\bibitem{chakravarty} Chakravarty, S.,  Tavoni, M.  Energy poverty alleviation and climate change mitigation: Is there a trade off? Energy Economics, 2013;40:S67--S73. https://doi.org/10.1016/j.eneco.2013.09.022

\bibitem{chei} Cheikh, N. B., Zaied, Y. B.,  Nguyen, D. K. Understanding energy poverty drivers in Europe. Energy Policy, 2023;183:113818.

\bibitem{R2} Chang K.-H., Chung H.-Y., Wang C.-N., Lai Y.-D., Wu, C.-H. A New Hybrid Fermatean Fuzzy Set and Entropy Method for Risk Assessment. Axioms 12:58, 2023.

\bibitem{deluca} de Luca, A.,  Termini, S. A definition of a nonprobabilistic entropy in the setting of fuzzy sets theory, Information and Computation, 1971;20:301--312.

\bibitem{desva} Desvallées, L. (2022). Low-carbon retrofits in social housing: Energy efficiency, multidimensional energy poverty, and domestic comfort strategies in southern Europe. Energy Research \& Social Science, 2022;85:102413.

\bibitem{dong} Dong, K., Dou, Y.,  Jiang, Q.  Income inequality, energy poverty, and energy efficiency: Who cause who and how?. Technological Forecasting and Social Change, 2022;179:121622.

\bibitem{ehsan} Ehsanullah, S., Tran, Q. H., Sadiq, M., Bashir, S., Mohsin, M.,  Iram, R.  How energy insecurity leads to energy poverty? Do environmental consideration and climate change concerns matters. Environmental Science and Pollution Research, 2021;28(39):55041--55052.

\bibitem{erdoganmelike} M. Erdogan, E. Ayyildiz, Comparison of hospital service performances under COVID-19 pandemics for pilot regions with low vaccination rates, Expert Systems with Applications, Volume 206, 15 November 2022, 117773


\bibitem{evans} Evans, J., Robinson, C., Davies, S.  Energy inefficiency as a ‘poverty premium’. Energy Research \& Social Science, 2024;118:103824.

\bibitem{fabbari} Fabbri, K.,  Gaspari, J.  Mapping the energy poverty: A case study based on the energy performance certificates in the city of Bologna. Energy and Buildings, 2021;234:110718.

%\bibitem{temelentropy} Gandotra, N.; Kizielewicz, B.; Anand, A.; Baczkiewicz, A.; Shekhovtsov, A.;Watrobski, J.; Rezaei, A.; Sałabun,W. New Pythagorean Entropy Measure with Application in Multi-Criteria Decision Analysis. Entropy 2021, 23, 1600.

%\bibitem{R3} Gandotra N., Kizielewicz B., Anand A., Baczkiewicz A., Shekhovtsov A., Watrobski J., Rezai A., Salabun W. New Pythagorean Entropy Measure with Application in Multi-Criteria Decision Analysis, Entropy, 23:1600, 2021.

\bibitem{gargetal} Garg H, Shahzadi G, Akram M. Decision-Making Analysis Based on Fermatean Fuzzy Yager Aggregation Operators with Application in COVID-19 Testing Facility. Mathematical Problems in Engineering, Volume 2020, Article ID 7279027, https://doi.org/10.1155/2020/7279027

\bibitem{guanetal}Guan, Y., Yan, J., Shan, Y., Zhou, Y., Hang, Y., Li, R., ...  Hubacek, K. Burden of the global energy price crisis on households. Nature Energy, 2023;8(3):304--316.

\bibitem{hasanuj} Hasanujzaman, M.,  Omar, M. A. Household and non-household factors influencing multidimensional energy poverty in Bangladesh: demographics, urbanization and regional differentiation via a multilevel modeling approach. Energy Research \& Social Science, 2022;92:102803.

\bibitem{hashem} Hasheminasab, H., Streimikiene, D.,  Pishahang, M. (2023). A novel energy poverty evaluation: Study of the European Union countries. Energy, 264, 126157

\bibitem{hills} Hills, J.  Getting the measure of fuel poverty: Final report of the Fuel Poverty Review. CASE Report 72, 2012, https://doi.org/10.13140/RG.2.1.4876.4244

\bibitem{jabri} Jabri, M., Flower, B.  Weight Perturbation: An Optimal Architecture and Learning Technique for Analog VLSI Feedforward and Recurrent Multilayered Networks, IEEE Trans. Neural Networks, 1992;3(1):154--157.

\bibitem{Jeevaraj} Jeevaraj S. Ordering of interval-valued Fermatean fuzzy sets and their applications. Expert Systems with Applications, 2021;185:115613.

\bibitem{jessel} Jessel, S., Sawyer, S.,  Hernández, D.  Energy, poverty, and health in climate change: a comprehensive review of an emerging literature. Frontiers in public health, 2019;7:357.

\bibitem{katoch} Katoch, O. R., Sharma, R., Parihar, S.,  Nawaz, A.  Energy poverty and its impacts on health and education: a systematic review. International Journal of Energy Sector Management, 2024;18(2):411--431.

\bibitem{kirisci2019} Kiri\c{s}ci M. Fibonacci statistical convergence on intuitionistic fuzzy normed spaces, Journal of Intelligent $\&$ Fuzzy Systems, 2019;36:5597--5604.

\bibitem{kirisciNew} Kiri\c{s}ci M., Correlation Coefficients of Fermatean Fuzzy Sets with Their Application, J. Math. Sci. Model. 2022;5(2):16--23. 

%\bibitem{Electre_paper} Kiri\c{s}ci M., Demir I., Simsek N. Fermatean fuzzy ELECTRE multi-criteria group decision-making and most suitable biomedical material selection. Artificial Intelligence in Medicine, 2022;127:102278.

\bibitem{kirisciIVFFL} Kiri\c{s}ci M. Data Analysis for Lung Cancer: Fermatean Hesitant Fuzzy Sets Approach, Applied Mathematics, Modeling and Computer Simulation, 2022;30.701--710.

\bibitem{kirisciKFCA} Kiri\c{s}ci M., Simsek N. A novel kernel principal component analysis with application disaster preparedness of hospital: interval-valued Fermatean fuzzy set approach, The Journal of Supercomputing 2023;79:19848--19878.

\bibitem{kirisci1} Kiri\c{s}ci M. Fermatean Fuzzy Type Statistical Concepts with Medical Decision-Making Application, Fuzzy Optimization and Modeling Journal 2023:4(1); 1--14.

\bibitem{kirisci0} Kiri\c{s}ci M. Data analysis for panoramic X-ray selection: Fermatean fuzzy type correlation coefficients approach, Engineering Applications of Artificial Intelligence 2023:126;106824.

%\bibitem{kirisci2} Kiri\c{s}ci M. New cosine similarity and distance measures for Fermatean fuzzy sets and TOPSIS approach, Knowledge and Information Systems 2023:65(2); 855--868.

%\bibitem{R4} Kiri\c{s}ci M. Medical Waste Management Based on an Interval-Valued Fermatean Fuzzy Decision-Making Method, Journal of Mathematical Sciences and Modelling, 2024;7(3):127--144.

\bibitem{R5} Kiri\c{s}ci M. Fermatean fuzzy type a three-way correlation coefficients. In: Gayoso Martinez, V., Yilmaz, F., Queiruga-Dios, A., Rasteiro, D.M., Martin-Vaquero, J., Mierlus¸-Mazilu, I. (eds) Mathematical Methods for Engineering Applications, ICMASE 2023. Springer Proceedings in Mathematics \& Statistics, 2024;499:325–338.

%\bibitem{kirisciwastewater} Kiri\c{s}ci M. The Risk Assessment of Wastewater Treatment with an Inte?grated Decision-Making Method, Proceedings of International Conference on Mathematics Advances and Applications, 2024;1:58-65.

\bibitem{kirisciASC}   Kiri\c{s}ci M. Interval-valued fermatean fuzzy based risk assessment for self-driving vehicles, Applied Soft Computing, \textbf{152} (2024), 111265.

\bibitem{kirisciMW} Kiri\c{s}ci M.  Medical Waste Management Based on an Interval-Valued Fermatean Fuzzy Decision-Making Method. Journal of Mathematical Sciences and Modelling. December 2024;7(3):128-145.

\bibitem{kirisci11}  Kiri\c{s}ci M. \emph{Fermatean Fuzzy Type a Three-Way Correlation Coefficients}. In: Gayoso Martínez, V., Yilmaz, F., Queiruga-Dios, A., Rasteiro, D.M., Martín-Vaquero, J., Mierluş-Mazilu, I. (eds) Mathematical Methods for Engineering Applications. ICMASE 2023. Springer Proceedings in Mathematics \& Statistics, 2024;499:325--338.

\bibitem{LiuLiu} Liu DH, Liu YY, Chen XH. Fermatean fuzzy linguistic set and its application in multicriteria decision-making.Int J Intell Syst. 2019;34(5):878‐894.


\bibitem{liuzhou} Liu, Z., Zhou, Z.,  Liu, C. (2023). Estimating the impact of rural centralized residence policy interventions on energy poverty in China. Renewable and Sustainable Energy Reviews, 2023;187:113687.

\bibitem{LuRen} Lu, S., Ren, J., Lee, C. K., Zhang, L.  Spatial-temporal energy poverty analysis of China from subnational perspective. Journal of Cleaner Production, 2022;341: 130907.

\bibitem{mafaldo} Mafalda Matos, A., Delgado, J. M.,  Guimarães, A. S.  Linking energy poverty with thermal building regulations and energy efficiency policies in Portugal. Energies, 2022;15(1):329.

\bibitem{mahu} Mahumane, G.,  Mulder, P. Urbanization of energy poverty? The case of Mozambique. Renewable and Sustainable Energy Reviews, 2022;159:112089.

\bibitem{okush} Okushima, S. Measuring energy poverty in Japan, 2004–2013. Energy policy, 2016;98:557--564.

\bibitem{qiChen} Qi, X., Chen, J., Wang, J., Liu, H.,  Ding, B. The impact of urbanization on the alleviation of energy poverty: Evidence from China. Cities, 2024;151:105130.

\bibitem{rahmanov} Rahnamay Bonab, S., Osgooei, E. Environment risk assessment of wastewater treatment using FMEA method based on Pythagorean fuzzy multiple-criteria decision-making. Environ Dev Sustain (2022). https://doi.org/10.1007/s10668-022-02555-5

\bibitem {ranMis1} Rani, P, Mishra, A R. Interval-valued fermatean fuzzy sets with multi-criteria weighted aggregated sum product assessment-based decision analysis Neural Computing and Applications, 2022:34;8051--8067


\bibitem{reames} Reames, T. G. (2016). Targeting energy justice: Exploring spatial, racial/ethnic and socioeconomic disparities in urban residential heating energy efficiency. Energy policy, 2016;97:549-558.

\bibitem{renner} Renner, S., Lay, J.,  Schleicher, M. The effects of energy price changes: heterogeneous welfare impacts and energy poverty in Indonesia. Environment and Development Economics, 2019;24(2):180--200.

\bibitem{riazetal} M. Riaz, H. M. A. Farid, H. M. Shakeel, D. Arif, \emph{Cost effective indoor HVAC energy efficiency monitoring based on intelligent decision support system under fermatean fuzzy framework}, Scientia Iranica \textbf{30}(6) (2023), 2143--2161.


\bibitem{robinson} Robinson, C.,  Mattioli, G.  Double energy vulnerability: Spatial intersections of domestic and transport energy poverty in the United Kingdom. Energy Research \& Social Science, 2020;70:101--120. https://doi.org/10.1016/j.erss.2020.101-120

\bibitem{SenYager} Senapati, T., Yager, R.R., 2019a. Fermatean fuzzy sets. J. Ambient Intell. Hum. Comput. 11, 663--674 (2020)

\bibitem{SenYager1} Senapati, T., Yager, R.R.  Some new operations over Fermatean fuzzy numbers and application of Fermatean fuzzy WPM in multiple criteria decision making. Informatica 30 (2), 2019, 391–412.

\bibitem{SenYager2} Senapati, T., Yager R.R. Fermatean fuzzy weighted averaging/geometric operators and its application in multi-criteria decision-making methods. Engineering Applications of Artificial Intelligence, 85 (2019) 112–121, https://doi.org/10.1016/j.engappai.2019.05.012

\bibitem{sheih} Shieh, H. S.,  Shah, S. A. A. (2025). Developing a Fuzzy MCDA-Based Multidimensional Index to Measure Energy Poverty in Developing Countries. Social Indicators Research, 2025;176(2):499--531.

\bibitem{sheng} Sheng, P., He, Y., Guo, X. The impact of urbanization on energy consumption and efficiency. Energy \& Environment, 2017;28(7):673--686.

\bibitem{kirisci22} Simsek, N., Kiri\c{s}ci M. Incomplete Fermatean Fuzzy Preference Relations and GroupDecision Making, Topological Algebra and its Applications 11 (1), 20220125

%\bibitem{kirisci33} Simsek, N., Kiri\c{s}ci M. A new risk assessment method for autonomous vehicle driving systems: Fermatean fuzzy AHP Approach. Istanbul Commerce University Journal of Science, 22(44), 2023, 292--309.

\bibitem{sova} Sovacool, B. K.,  Furszyfer Del Rio, D. D. Smart home technologies in Europe: A critical review of concepts, benefits, risks and policies. Renewable and Sustainable Energy Reviews, 2020;120:109-118. https://doi.org/10.1016/j.rser.2019.109-118

\bibitem{streim} Streimikiene, D., Lekavičius, V., Baležentis, T., Kyriakopoulos, G. L.,  Abrhám, J.  Climate change mitigation policies targeting households and addressing energy poverty in European Union. Energies, 2020;13(13):3389.

\bibitem{sun}  Sun, S., Tong, K. Rural-urban inequality in energy use sufficiency and efficiency during a rapid urbanization period. Applied Energy, 2024;364:123133.

\bibitem{teix} Teixeira, I., Ferreira, A. C., Rodrigues, N.,  Teixeira, S. Energy Poverty and Its Indicators: A Multidimensional Framework from Literature. Energies, 2024;17(14):3445.

\bibitem{thompson} Thomson, H., Snell, C.,  Bouzarovski, S. Health, well-being and energy poverty in Europe: A comparative study of 32 European countries. International journal of environmental research and public health, 2017;14(6):584.

\bibitem{thompsenetal} Thomson, H., Simcock, N., Bouzarovski, S.,  Petrova, S.  Energy poverty and indoor cooling: An overlooked issue in Europe. Energy and Buildings, 2019;196:21--29.

\bibitem{thompsnell} Thomson, H.,  Snell, C.  Quantifying the prevalence of fuel poverty across the European Union. Energy Policy, 2013;52:563--572. https://doi.org/10.1016/j.enpol.2012.10.009

\bibitem{thompsonBou} Thomson, H.,  Bouzarovski, S. Addressing energy poverty in the European Union: State of play and action. European Commission, 2018.

\bibitem{xuetal} Xu, Q., Khan, S., Zhang, X.,  Usman, M.  Urbanization, rural energy-poverty, and carbon emission: unveiling the pollution halo effect in 48 BRI countries. Environmental Science and Pollution Research, 2023;30(48):105912--105926.

\bibitem{wang} Wang, Y., Qiao, G., Ahmad, M.,  Yang, D.  Modeling the impact of fiscal decentralization on energy poverty: do energy efficiency and technological innovation matter?. International Journal of Environmental Research and Public Health, 2023;20(5):4360.

\bibitem{Yager0}
Yager RR. Pythagorean fuzzy subsets.   Proc. Joint IFSA World Congress and NAFIPS Annual Meeting,
Edmonton, Canada, 2013.

\bibitem{Yager} Yager, R.R. Pythagorean membership grades in multicriteria decision-making. IEEE Transactions on Fuzzy Systems 22(4) (2014) 958--965.

\bibitem{yadava} Yadava, R. N.,  Sinha, B. Energy–poverty–climate vulnerability nexus: an approach to sustainable development for the poorest of poor. Environment, Development and Sustainability, 2024;26(1):2245--2270.

\bibitem{yawele} Yawale, S. K., Hanaoka, T., Kapshe, M.,  Pandey, R. End-use energy projections: Future regional disparity and energy poverty at the household level in rural and urban areas of India. Energy Policy, 2023;182:113772.


\bibitem{Zadeh} Zadeh LA. Fuzzy sets. Inf Comp 1965;8:338–353

\bibitem{zhouwang} Zhou, K., Wang, Y.,  Hussain, J.  Energy poverty assessment in the Belt and Road Initiative countries: Based on entropy weight-TOPSIS approach. Energy Efficiency, 2022;15(7):46.



%%%%%%%%%%%%%%%%%%%%%%%%%%%%%%%%%%%%%




%----------------------


\end{thebibliography}
\end{document}